\documentclass[11pt]{article}

\usepackage{amsmath,amssymb,amsfonts,textcomp,amsthm,xifthen,psfrag,graphicx,color,MnSymbol}
\usepackage[T1]{fontenc}

\usepackage{hyperref}

\usepackage{pgfplots}
\pgfplotsset{compat=newest}
\usepgfplotslibrary{external}
\date{}

\newtheorem{theorem}{Theorem}
\newtheorem{lemma}[theorem]{Lemma}

\newtheorem{prop}[theorem]{Proposition}
\newtheorem{remark}[theorem]{Remark}
\theoremstyle{definition} 

\newcommand{\<}{\langle{}}
\renewcommand{\>}{\rangle}

\newcommand{\ip}[2]{\llangle#1\hspace*{.5mm},#2\rrangle}
\newcommand{\dual}[2]{\<#1\hspace*{.5mm},#2\>}
\newcommand{\vdual}[2]{(#1\hspace*{.5mm},#2)}

\newcommand{\norm}[2]{\|#1\|_{#2}}

\newcommand{\diam}{\mathrm{diam}}
\newcommand{\wilde}{\widetilde}
\newcommand{\wat}{\widehat}

\def\Grad{\boldsymbol{\varepsilon}}
\def\pwGrad{\boldsymbol{\varepsilon}_\cT}
\def\Div{{\rm\bf div\,}}
\def\pwDiv{{\rm\bf div}_\cT}
\def\grad{\nabla}
\def\pwgrad{\nabla_\cT}
\def\MM{\mathbf{M}}
\def\QQ{\mathbf{Q}}

\def\TTheta{\mathbf{\Theta}}
\def\bq{\boldsymbol{q}}

\newcommand{\bL}{\ensuremath{\mathbf{L}}}
\newcommand{\LL}{\ensuremath{\mathbb{L}}}
\newcommand{\PP}{\ensuremath{\mathbb{P}}}
\def\tQ{\wat{\boldsymbol{q}}}

\def\tu{\wat{\boldsymbol{u}}}
\def\tv{\wat{\boldsymbol{v}}}
\newcommand{\EE}{\ensuremath{\mathbb{E}}}
\newcommand{\UU}{\ensuremath{\mathcal{U}}}
\newcommand{\VV}{\ensuremath{\mathcal{V}}}

\newcommand{\bP}{\ensuremath{\mathbf{P}}}

\newcommand{\br}{\ensuremath{\mathbf{r}}}
\newcommand{\bj}{\ensuremath{\mathbf{j}}}

\newcommand{\bu}{\ensuremath{\mathbf{u}}}

\newcommand{\bv}{\boldsymbol{v}}
\newcommand{\vv}{\ensuremath{\mathbf{v}}}
\newcommand{\uu}{\ensuremath{\mathbf{u}}}

\newcommand{\bw}{\ensuremath{\mathbf{w}}}

\newcommand{\bphi}{\ensuremath{\boldsymbol{\phi}}}
\newcommand{\PPhi}{\ensuremath{\boldsymbol{\Phi}}}

\newcommand{\bvarphi}{\ensuremath{\boldsymbol{\varphi}}}

\newcommand{\deltaz}{\delta\!z}
\newcommand{\Deltaq}{\boldsymbol{\delta}\!\QQ}
\newcommand{\bdeltaq}{\boldsymbol{\delta}\!\bq}
\newcommand{\bdeltau}{\boldsymbol{\delta}\!\bu}

\newcommand{\traceDD}[1]{\mathrm{tr}_{#1}^{\mathrm{dDiv}}} %{-3/2,-1/2}}
\newcommand{\traceGG}[1]{\mathrm{tr}_{#1}^{\mathrm{Ggrad}}} %{-3/2,-1/2}}

\newcommand{\bH}{\ensuremath{\mathbf{H}}}
\newcommand{\DDs}{\ensuremath{\mathbb{D}^s}}
\newcommand{\bD}{\ensuremath{\boldsymbol{\mathcal{D}}}}
\newcommand{\cD}{\ensuremath{\mathcal{D}}}

\newcommand{\trggrad}[1]{{\mathrm{Ggrad},#1}}
\newcommand{\trddiv}[1]{{\mathrm{dDiv},#1}}

\def\tr{\mathrm{tr}}

\def\div{{\rm div\,}}
\def\divwo{{\rm div}}
\def\pwdiv{ {\rm div}_\cT\,}

\newcommand{\jump}[1]{[#1]}
\newcommand{\jjump}[1]{\lsem #1\rsem}

\newcommand{\ttt}{{\rm T}}

\newcommand{\di}{d}
\newcommand{\R}{\ensuremath{\mathbb{R}}}
\newcommand{\N}{\ensuremath{\mathbb{N}}}
\newcommand{\cH}{\ensuremath{\mathcal{H}}}
\newcommand{\nn}{\ensuremath{\mathbf{n}}}
\newcommand{\kkappa}{\ensuremath{\boldsymbol{\kappa}}}

\newcommand{\cC}{\ensuremath{\mathcal{C}}}
\newcommand{\cCinv}{\ensuremath{\mathcal{C}^{-1}}}
\newcommand{\cT}{\ensuremath{\mathcal{T}}}

\newcommand{\cS}{\ensuremath{\mathcal{S}}}

\newcommand{\bt}{\ensuremath{\mathbf{t}}}

\newcommand{\OO}{\ensuremath{\mathcal{O}}}
\newcommand{\cE}{\ensuremath{\mathcal{E}}}
\newcommand{\cN}{\ensuremath{\mathcal{N}}}

\title{An ultraweak formulation of the Kirchhoff--Love plate bending model and DPG approximation
\thanks{Supported by CONICYT through FONDECYT projects 1150056, 11170050, The Magnus Ehrnrooth foundation,
        and by Oulun rakennustekniikan s\"a\"ati\"o.}}
\author{
Thomas~F\"uhrer$^\dagger$
\and
Norbert Heuer\thanks{
Facultad de Matem\'aticas, Pontificia Universidad Cat\'olica de Chile,
Avenida Vicu\~na Mackenna 4860, Santiago, Chile,
email: {\tt \{tofuhrer,nheuer\}@mat.uc.cl}}
\and
Antti H. Niemi\thanks{
Structures and Construction Technology Research Unit, Faculty of Technology, University of Oulu,
Erkki Koiso-Kanttilan katu 5, Linnanmaa, 90570 Oulu, Finland,
email: {\tt antti.niemi@oulu.fi}}}

\begin{document}
\maketitle
\begin{abstract}
We develop and analyze an ultraweak variational formulation for a variant
of the Kirchhoff--Love plate bending model.
Based on this formulation, we introduce a discretization of the discontinuous
Petrov--Galerkin type with optimal test functions (DPG).
We prove well-posedness of the ultraweak formulation and quasi-optimal convergence
of the DPG scheme.

The variational formulation and its analysis require tools that control traces and jumps
in $H^2$ (standard Sobolev space of scalar functions) and $H(\div\Div\!)$ (symmetric tensor functions
with $L_2$-components whose twice iterated divergence is in $L_2$), and their dualities.
These tools are developed in two and three spatial dimensions.
One specific result concerns localized traces in a dense subspace of $H(\div\Div\!)$.
They are essential to construct basis functions for an approximation of $H(\div\Div\!)$.

To illustrate the theory we construct basis functions of the lowest order and perform
numerical experiments for a smooth and a singular model solution.
They confirm the expected convergence behavior of the DPG method
both for uniform and adaptively refined meshes.


\end{abstract}

%===================================================================================================
\section{Introduction}

We develop an ultraweak variational formulation for a bending-moment variant of the Kirchhoff--Love
plate model, and present a discontinuous Petrov--Galerkin method with optimal test functions
(DPG method) that is based on this formulation. We prove well-posedness of the continuous
formulation and quasi-optimal convergence of the discrete scheme.
At the heart of the analysis is the space $H(\div\Div\!,\Omega)$ and its traces and jumps.
This space consists of symmetric tensors with $L_2(\Omega)$-components whose
twice iterated divergence is in $L_2(\Omega)$
(the notation $\Div$ indicates the divergence operator that acts on the rows of tensors).

The Kirchhoff--Love model was introduced by Kirchhoff \cite{Kirchhoff_50_UGB}
in a form that is generally accepted today.
Kirchhoff also applied the model to determine the free vibration frequencies and modes of circular plates.
A historical account of the development of the model is incorporated in \cite{Love_88_SFV}
where Love uses Kirchhoff's approach to study vibrations of initially curved shells.
Nowadays, the model is widely used in structural engineering,
e.g.,~to dimension reinforced concrete slabs under static loads \cite{HausslerCombe_15_CMR}
and to control disturbing vibrations of wooden floors and other lightweight plane structures.

Perhaps the most well-known mathematical representation of the Kirchhoff--Love model for linearly elastic
and isotropic material is given by the biharmonic equation
\begin{equation*} %\label{biharmonic}
   D\Delta^2 u = f,
\end{equation*}
where $u:\;\Omega\to\R$ is the deflection of the plate mid-surface $\Omega\subset\R^2$,
$\Delta$ is the Laplace operator and $f:\;\Omega\to\R$ and $D>0$ represent the external loading
and bending rigidity of the plate, respectively.  

It is evident that application of the model to complex geometries requires employment of numerical methods
such as the finite element method.
The literature on the numerical analysis of plate bending problems is vast due to the aforementioned
practical relevance of the problems and respectable age of the structural models.
It is not feasible to perform a thorough literature review here but two points that motivate
our work can be made.
First, conventional methods based on the variational principle of virtual displacements produce as
direct output only the deflection values. These, albeit needed values, are not sufficient for
structural design purposes where stresses and their resultants are of utmost importance.
Second, verification of numerical accuracy of finite element algorithms is at the hearth
of simulation governance, see~\cite{SzaboA_12_SGT}.
This is a serious challenge in practical plate-bending problems where both the geometry and
applied loading can be very irregular so that many of the contemporary developments in
the finite element modeling of plate problems are devoted to a posteriori error estimation
and adaptivity, see, e.g.,~\cite{Carstensen_02_RBP,BeiraodaVeigaNS_10_PEE,HansboL_11_PEE}.

We develop the theoretical framework for a DPG discretization to address the above challenges
and, perhaps more significant, to set a theoretical basis to develop and analyze DPG schemes
for other structural models like the singularly perturbed Reissner--Mindlin plate model
and different shell models.
Our analysis includes the case of singular problems on non-convex plates in contrast to
many publications that assume convexity or smooth boundaries. In this context,
we mention the mixed formulation from Amara \emph{et al.} \cite{AmaraCPC_02_BMM} who
specifically use the space $H(\div\Div\!,\Omega)$ (without symmetry), thus allowing
for singularities. Their numerical scheme is based on a decomposition of
$H(\div\Div\!,\Omega)$ resulting in a mixed formulation that can be discretized by
standard finite elements. In \cite{Gallistl_17_SSP}, Gallistl proposes a similar splitting
approach for polyharmonic problems with corresponding finite element scheme.

The DPG framework has been founded by Demkowicz and Gopalakrishnan in \cite{DemkowiczG_11_CDP}.
It is very flexible and can be used with various variational formulations.
A posteriori error estimation is also built-in, see \cite{DemkowiczGN_12_CDP}.
DPG schemes have been applied previously to structural engineering
problems in \cite{NiemiBD_11_DPG,CaloCN_14_ADP}
and to more general problems of elasticity in \cite{BramwellDGQ_12_Lhp,KeithFD_17_DMA}.
The most closely related investigation to the present work is probably \cite{CaloCN_14_ADP}.
That investigation showed that an ultraweak variational formulation of the Reissner--Mindlin
plate bending model is well posed and that the associated discretization is convergent.
Rather accurate numerical results were observed despite the fact that the theoretically
obtained stability constant is very weak and depends on the slenderness of the plate.
In particular, the question of well-posedness of the ultraweak variational formulation of the
asymptotic Kirchhoff--Love model was left open.

Essential motivation for the use of DPG schemes is their possible robustness for singularly perturbed
problems. The intrinsic energy norm can bound in a robust way approximation errors in the sense
that quasi-optimal error estimates (by the energy norm, which is accessible) are uniform with
respect to perturbation parameters. To achieve this robustness in appropriate (selected) norms,
it is paramount to have an appropriate variational formulation, and proving robustness is usually non-trivial.
For an analysis of second-order elliptic problems with convection-dominated diffusion (``confusion'')
and reaction-dominated diffusion (``refusion'') we refer to
\cite{DemkowiczH_13_RDM,ChanHBTD_14_RDM,BroersenS_14_RPG,BroersenS_15_PGD}
and \cite{HeuerK_17_RDM}, respectively.
The DPG setting for refusion from \cite{HeuerK_17_RDM} has been extended
to transmission problems and the coupling with boundary elements \cite{FuehrerH_17_RCD},
and to Signorini-type contact problems \cite{FuehrerHS_DMS}.

Whereas we do not consider a singularly perturbed problem in this paper, the development
of a DPG scheme for the Kirchhoff--Love model is relevant in its own right as discussed before, and will
be essential to deal with other models of plate problems. Since we expect our technical
tools to be useful also for fourth-order problems in three dimensions, they are developed
for both two and three space dimensions (they can be generalized to any space dimension).
Discretizations of fourth-order problems usually avoid $H^2$-bilinear forms to employ
simpler than $H^2$-conforming basis functions. In this respect, our choice of ultraweak
variational formulation has the advantage that field variables are only in $L_2$-spaces
whereas appearing trace variables (traces of $H^2(\Omega)$ and $H(\div\Div\!,\Omega)$)
are relatively straightforward to discretize.

Let us discuss the structure of our work. In the next section we introduce the model problem
of a certain bending-moment formulation for the Kirchhoff--Love model.
For simplicity we assume fully clamped plates but this is not essential as our formulation
gives access to all kinds of boundary conditions. In that section, we also start developing a
variational formulation. Since DPG schemes use product
spaces\footnote{Often they are referred to as ``broken'' spaces. We prefer to call
them product spaces since important Sobolev spaces, e.g., of negative order or order $1/2$,
cannot be localized but have to be defined as product spaces from the start, cf.~\cite{HeuerK_17_DPG}.}
with respect to subdivisions of $\Omega$ into elements, trace operations in the underlying Sobolev spaces
appear naturally. For fourth-order problems this is a non-trivial issue. Therefore,
in order to define a well-posed variational formulation in product spaces we need to develop
trace and jump operations, in our case in $H^2(\Omega)$ and $H(\div\Div\!,\Omega)$.
This is subject of Section~\ref{sec_traces_jumps}, whose contents is discussed in more detail below.
Eventually, in Section~\ref{sec_uw}, we are able to define our ultraweak variational
formulation and state its well-posedness (Theorem~\ref{thm_stab}). We then briefly
define the DPG scheme and state its quasi-optimal convergence (Theorem~\ref{thm_DPG}).
Proofs of Theorems~\ref{thm_stab} and~\ref{thm_DPG} are given in Section~\ref{sec_adj}.
We do not dwell much on the discussion of DPG schemes and their analysis. It is known that
an analysis of the underlying adjoint problem gives access to the well-posedness of the
variational formulation and quasi-optimal convergence of the DPG method (references have
been given above). Though we do stress the fact that our analysis goes beyond standard
techniques. Rather than splitting the adjoint problem into a homogeneous one in product spaces
and an inhomogeneous one in global (``unbroken'' or non-product) spaces
(like, e.g., in \cite{DemkowiczG_11_ADM,DemkowiczH_13_RDM,HeuerK_17_RDM})
or deducing stability of the adjoint problem in product spaces from the one of the global form
\cite{CarstensenDG_16_BSF}, we consider the full adjoint problem as a whole.
Section~\ref{sec_adj} starts with defining the adjoint problem. Its well-posedness is
proved in \S\ref{sec_adj_well}.  Key idea is to describe the primal unknown of the
adjoint problem as the solution to a saddle point problem without Lagrange multiplier.
Specifically, the primal unknown stays in the original product space and test functions
are considered in the corresponding global space. Of course, this problem
could be reformulated as a traditional saddle point problem. However, our technique
is applicable to adjoint problems with data that require
continuity,\footnote{Here we only note that such restrictions appear when considering first-order
formulations of plate bending models.}
that is, leaving the $L_2$ setting of ultraweak formulations.
In this sense, our new technique of analyzing the adjoint problem is fundamental.
Extensions to other problems will be subject of future research.

Let us note that there is a recent abstract framework by Demkowicz \emph{et al.}
\cite{DemkowiczGNS_17_SDM}. Under specific assumptions it yields the well-posedness
of $L_2$-ultraweak formulations in product spaces without explicitly analyzing trace spaces.
In \cite{GopalakrishnanS_SDM}, Gopalakrishnan and Sep\'ulveda applied this setting to
acoustic wave problems. In both references, an essential density assumption
is only proved for simple geometries. Furthermore, trace variables are discretized
via their domain counterparts whereas we only discretize the traces.
It is also unknown whether the new framework gives robust control
of variables in the case of singularly perturbed problems.
In \cite{ErnestiW_STD,Ernesti_thesis}, Ernesti and Wieners presented a simplified
DPG analysis based on the framework from \cite{DemkowiczGNS_17_SDM}.
They use the density results for simple geometries from
\cite{DemkowiczGNS_17_SDM,GopalakrishnanS_SDM}. Furthermore, the construction of
their trace discretization is done without explicitly defining the domain parts,
although they are needed for the stability and approximation analysis.
In conclusion, in comparison with the current state of the framework from
\cite{DemkowiczGNS_17_SDM}, our strategy has the advantages of giving access to
singularly perturbed problems, being extendable to non-$L_2$ settings,
avoiding domain contributions for trace discretizations, and not requiring
density assumptions which can be hard to prove
(though, see \cite[Proposition 2.1]{AmaraCPC_02_BMM} for the density of smooth tensor functions
in $H(\div\Div\!,\Omega)$ defined by the graph norm without symmetry).

Now, to continue discussing the contents of our paper,
having the analysis of the adjoint problem from \S\ref{sec_adj_well} at hand,
the proofs of Theorems~\ref{thm_stab} and~\ref{thm_DPG} are straightforward.
They are given in \S\ref{sec_pf}.
Finally, in Section~\ref{sec_num} we discuss the construction of discrete spaces for
our DPG scheme and give some numerical examples. \S\ref{sec:discretespaces}
is devoted to the construction of lowest-order basis functions. Whereas the field variables
do not require any continuity across element interfaces, it is more technical to
identify unknowns associated with trace variables.
Specifically, the construction of basis functions for traces of
$H(\div\Div\!,\Omega)$ requires to identify \emph{local} continuity constraints. It turns out
that traces of $H(\div\Div\!,\Omega)$-functions cannot be split into natural components
that allow for such a construction. This is analogous to $H(\div\!,\Omega)$ where one
uses a slightly more regular subspace of vector functions with normal (then localizable)
traces in $L_2$. In the literature, this subspace is usually denoted by $\cH(\div\!,\Omega)$.
In $H(\div\Div\!,\Omega)$ the situation is worse since the definition of
traces requires to integrate by parts \emph{twice}. This generates two combined traces.
We present lowest-order basis functions (for traces of $H(\div\Div\!,\Omega)$)
that correspond to local unknowns associated with edges and nodes of triangular elements,
plus jump constraints associated with interior nodes and neighboring elements.
These constraints can be imposed by Lagrange multipliers.
For sufficiently smooth solutions, our lowest order scheme
converges with optimal order (Theorem~\ref{thm:approxU}).
This result assumes the use of \emph{optimal test functions} whereas, obviously,
our numerical implementation uses approximated optimal test functions.
We do not analyze the influence of this approximation here.
In \S\ref{sec_num_ex} we present numerical results for two examples, the case
of a smooth solution and the case of a singular solution.
Uniform mesh refinement yields optimal and sub-optimal convergence, respectively,
whereas an adaptive variant restores optimal convergence for the singular example.
It is worth mentioning that the singular example solution generates a tensor of
$H(\div\Div\!,\Omega)$ whose divergence is less than $L_2$-regular.
This shows, in particular, that our analysis of traces and jumps in $H(\div\Div\!,\Omega)$
cannot be split into two steps/spaces (symmetric tensors in $\bH(\Div\!,\Omega)$
whose divergence are elements of $H(\div\!,\Omega)$).

To conclude, the central focus of this paper is on the analysis of traces and jumps in
$H(\div\Div\!,\Omega)$, in Section~\ref{sec_traces_jumps}.
Despite of considering a plate model, this analysis is done in two and three space dimensions.
It is relevant for other fourth-order problems in three dimensions.
Section~\ref{sec_traces_jumps} is split into several subsections.
In the first two, \S\S\ref{sec_trace_dd} and \ref{sec_trace_gg}, we
define and analyze trace operators in $H(\div\Div\!,\Omega)$ and $H^2_0(\Omega)$
(denoted by $\traceDD{}$ and $\traceGG{}$, with local versions
$\traceDD{T}$ and $\traceGG{T}$, respectively), and corresponding trace spaces and norms.
In \S\ref{sec_jump_dd}, we consider the product variant $H(\div\Div\!,\cT)$ of $H(\div\Div\!,\Omega)$
and jumps of its elements. Specifically, we characterize the inclusion
$H(\div\Div\!,\cT)\subset H(\div\Div\!,\Omega)$ through (vanishing) duality with $H^2_0(\Omega)$
(Proposition~\ref{prop_dd_jump}). In \S\ref{sec_jump2_dd} we revisit (a subspace of)
the product space $H(\div\Div\!,\cT)$ and study traces rather than jumps (of course, trace operators can
be used to define and analyze jumps).
We define a dense product subspace $\cH(\div\Div\!,\cT)\subset H(\div\Div\!,\cT)$
and prove that our previous ``local'' trace operators $\traceDD{T}$
(they act on boundaries of elements) can be further localized when acting on this subspace
(Proposition~\ref{prop_cor_dd_jump}).
This is of utmost importance for the numerical scheme since it implies density of our discrete spaces
in $H(\div\Div\!,\Omega)$, and thus convergence.
\S\ref{sec_jump_gg} corresponds to \S\ref{sec_jump_dd}, considering
jumps of a product space $H^2(\cT)$ rather than of $H(\div\Div\!,\cT)$, with continuity
characterization by duality with the trace space $\traceDD{}(H(\div\Div\!,\Omega))$
(Proposition~\ref{prop_gg_jump}).

The final Subsection~\ref{sec_gg_Poincare} provides a Poincar\'e inequality
in the product space $H^2(\cT)$. Recall that traditional stability proofs of adjoint
problems separate the analysis into a global non-homogeneous problem and a homogeneous
one in product spaces and with jump data. The non-homogeneous problem usually gives control of
a seminorm of the primal variable so that a Poincar\'e inequality is required to bound
the norm. Furthermore, proving stability of homogeneous adjoint problems with jump data
is usually done via a Helmholtz decomposition.
For details see, e.g., \cite[Lemmas 4.2, 4.3]{DemkowiczG_11_ADM}.
In our case, the global adjoint problem gives also only access to a seminorm of the primal variable,
and still, the connection between jump data and the field variable is established by a Helmholtz
decomposition. We combine both techniques and give a short proof of a Poincar\'e inequality
in $H^2(\cT)$ which uses a Helmholtz decomposition only implicitly.

Throughout the paper, $a\lesssim b$ means that $a\le cb$ with a generic constant $c>0$ that is independent of
the underlying mesh (except for possible general restrictions like shape-regularity of elements).
Similarly, we use the notation $a\simeq b$ and $a\gtrsim b$.

%===================================================================================================
\section{Model problem} \label{sec_model}

We start by recalling the Kirchhoff--Love model, cf.~\cite{VentselK_01_TPS}.
The static variables of the model are the shear force vector $\QQ$ and the symmetric bending
moment tensor $\MM$.
These stand for stress resultants representing internal forces and moments per unit length
along the coordinate lines on the plate mid-surface $\Omega$.
They are related to the external surface load $f$ and to each other by the laws of static
equilibrium (force and moment balance) as 
\begin{alignat*}{2}
   -\div \QQ &= f  && \quad\text{in} \quad \Omega,\\ %\label{m1},\\
   \QQ &= \Div \MM && \quad\text{in} \quad \Omega.   %\label{m2}.
\end{alignat*}
The operator $\div$ denotes the divergence of vector functions, and $\Div$ is the divergence
operator acting on rows of tensors.
Denoting by $\Grad$ the infinitesimal strain tensor, or symmetric gradient, we introduce
the bending curvature $\kkappa = \Grad(\grad u)
:=\frac 12(\boldsymbol{\grad}(\grad u)+\boldsymbol{\grad}(\grad u)^T)$, the Hessian of $u$ in our case.
For linearly elastic isotropic material, the bending moments can be written in terms of $\kkappa$ as
\begin{equation*}
   \MM = -\cC\kkappa = -D [\nu\,\tr \kkappa \mathbf{I} + (1-\nu) \kkappa]
\end{equation*}
where 
\[
   D= \frac{Et^3}{12(1-\nu^2)} 
\]
is the bending rigidity of the plate defined in terms of the Young modulus $E$ and
Poisson ratio $\nu$ of the material and the plate thickness $t$.
The values of these parameters are not very critical concerning the numerical solution of the
problem. $D$ acts as scaling parameter and the influence of the Poisson ratio on the solution is mild.
We select fixed $\nu\in(-1,1/2]$ and $t>0$ so that $\cC$ is positive definite.

Let us now assume that $\Omega\subset\R^\di$ ($\di=2,3$) is a bounded simply connected Lipschitz domain
with boundary $\Gamma=\partial\Omega$.
(Of course, for the plate-bending problem, only $\di=2$ is physically motivated.)
For a given $f\in L_2(\Omega)$ our model problem is
\begin{subequations} \label{prob}
\begin{alignat}{2}
     -\div\Div\MM               &= f  && \quad\text{in} \quad \Omega\label{p1},\\
    \cCinv \MM + \Grad\grad u   &= 0  && \quad\text{in} \quad \Omega\label{p3},\\
    u = 0,\quad \nn\cdot\grad u &= 0 &&\quad\text{on}\quad\Gamma.\label{pBC}
\end{alignat}
\end{subequations}
Here, $\nn$ is the exterior unit normal vector on $\Gamma$. Later, $\nn$ will be used generically
for normal vectors. Before starting to develop a variational formulation, we introduce a mesh $\cT$
that consists of general non-intersecting open Lipschitz elements.
Only in \S\ref{sec_jump2_dd} we will require that the mesh is conforming and consists of
generalized (curved) polyhedra/polygons, and in the numerical section \S\ref{sec_num}
we restrict ourselves to two space dimensions and conforming triangular meshes of shape-regular
elements. To the mesh $\cT=\{T\}$ we associate the skeleton $\cS=\{\partial T;\;T\in\cT\}$.
For $T\in\cT$, scalar functions $z$ and symmetric tensors $\TTheta$, let us define the norms
\begin{align*}
   \|z\|_{2,T}^2 &:= \|z\|_T^2 + \|\Grad\grad z\|_T^2,\quad
   \|\TTheta\|_{\div\Div\!,T}^2 := \|\TTheta\|_T^2 + \|\div\Div\TTheta\|_T^2,
\end{align*}
and induced spaces $H^2(T)$, $H(\div\Div\!,T)$ as closures of
$\cD(\overline{T})$ and $\DDs(\overline{T})$ with respect to the corresponding norm.
Here, $\cD(\overline{T})$ and $\DDs(\overline{T})$ are the spaces of smooth functions
and smooth symmetric tensors, respectively, on $T$.
(The logic for the notation $H(\div\Div\!,T)$ with plain letter $H$ is that tensors are mapped
to scalar functions by the operator $\div\Div\!$. Similarly, below we introduce $\bH(\Div\!,T)$
with bold $\bH$ as $\Div\!$ maps tensor functions to vector functions.)
Throughout the paper, $\|\cdot\|_\omega$ denotes the $L_2(\omega)$-norms for scalar, vector and
tensor functions on the indicated set $\omega$.
When $\omega=\Omega$ we drop the index and simply write $\|\cdot\|$ instead of $\|\cdot\|_\Omega$.
The corresponding bilinear forms are $\vdual{\cdot}{\cdot}_\omega$ and $\vdual{\cdot}{\cdot}$.
The spaces $\LL_2^s(\Omega)$ and $\LL_2^s(T)$ denote symmetric tensor functions on $\Omega$
and $T$, respectively.

Now, given a mesh $\cT$, we define product spaces
(tacitly identifying product spaces with their broken variants)
\begin{align*}
   H^2(\cT)      &:= \{z\in L_2(\Omega);\; z|_T\in H^2(T)\ \forall T\in\cT\},\\
   H(\div\Div\!,\cT) &:= \{\TTheta\in\LL_2^s(\Omega);\; \TTheta|_T\in H(\div\Div\!,T)\}
\end{align*}
with canonical product norms $\|\cdot\|_{2,\cT}$ and $\|\cdot\|_{\div\Div\!,\cT}$, respectively.
We will also need the global spaces $H^2_0(\Omega)$ and $H(\div\Div\!,\Omega)$
which are the closures of $\cD(\Omega)$ and $\DDs(\overline{\Omega})$, respectively, with
corresponding norms $\|z\|_2^2=\|z\|^2+\|\Grad\grad z\|^2$ and
$\|\TTheta\|_{\div\Div}^2=\|\TTheta\|^2+\|\div\Div\TTheta\|^2$, and similarly $H^2_0(T)$ for $T\in\cT$.

Now, we test
\begin{align*}
   &\eqref{p1}\text{ with } z\in H^2(\cT),
   \quad\text{and}\quad
   \eqref{p3}\text{ with }\TTheta\in H(\div\Div\!,\cT).
\end{align*}
Formally integrating by parts on every element $T\in\cT$ and summing over the elements
and summing the two equations, the testing results in
\begin{equation} \label{VFa}
\begin{split}
     \vdual{\MM}{\pwGrad\pwgrad z}
   + \sum_{T\in\cT} \dual{\nn\cdot\Div\MM}{z}_{\partial T}
   - \sum_{T\in\cT} \dual{\MM\nn}{\grad z}_{\partial T}
   \\
   + \vdual{\cCinv\MM}{\TTheta} + \vdual{u}{\pwdiv\pwDiv\TTheta}
   + \sum_{T\in\cT} \dual{\TTheta\nn}{\grad u}_{\partial T}
   - \sum_{T\in\cT} \dual{\nn\cdot\Div\TTheta}{u}_{\partial T}
     &= -\vdual{f}{z}.
\end{split}
\end{equation}
Here and in the following, a differential operator with index $\cT$ means that it is taken
piecewise with respect to the elements $T\in\cT$. We will write equivalently,
e.g., $\vdual{\MM}{\pwGrad\pwgrad z}=\vdual{\MM}{\Grad\grad z}_\cT$, and similarly for other
differential operators taken in a piecewise form.
Furthermore, we use the generic notation $\nn$ for the unit normal vector on $\partial T$ and $\Gamma$, pointing
outside $T$ and $\Omega$, respectively.
The notation $\dual{\cdot}{\cdot}_{\omega}$, and later $\dual{\cdot}{\cdot}_\Gamma$,
indicate dualities on $\omega\subset\partial T$ and $\Gamma$, respectively, with $L_2$-pivot space.

At this point it is not clear whether the appearing normal components in \eqref{VFa}
on the boundaries of elements are well defined. Indeed, essential part of this paper is to
study the relation between traces and jumps of the involved spaces $H^2(\cT)$, $H^2_0(\Omega)$,
$H(\div\Div\!,\cT)$ and $H(\div\Div\!,\Omega)$. This will be done in the next section, before
returning to a variational formulation of \eqref{prob} in Section~\ref{sec_uw}.

%===================================================================================================
\section{Traces, jumps and a Poincar\'e inequality} \label{sec_traces_jumps}

In the following we introduce and analyze operators and norms that serve to give the terms
$\nn\cdot\Div\MM|_{\partial T}$, $\MM\nn|_{\partial T}$,
$\nn\cdot\Div\TTheta|_{\partial T}$, and $\TTheta\nn|_{\partial T}$,
from \eqref{VFa} a meaning for $\MM\in H(\div\Div\!,\Omega)$ and $\TTheta\in H(\div\Div\!,\cT)$.

%===================================================================================================
\subsection{Traces of $H(\div\Div\!,\Omega)$} \label{sec_trace_dd}

We start by introducing linear operators $\traceDD{T}:\;H(\div\Div\!,T)\to H^2(T)'$
for $T\in\cT$ by
\begin{align} \label{trT_dd}
   \dual{\traceDD{T}(\TTheta)}{z}_{\partial T} := \vdual{\div\Div\TTheta}{z}_T - \vdual{\TTheta}{\Grad\grad z}_T.
\end{align}
We note that this definition is consistent with the observation made by Amara \emph{at al.} in
\cite[Theorem~2.2]{AmaraCPC_02_BMM} (they consider the whole domain $\Omega$ instead of an element $T$).
The range of the operator $\traceDD{T}$ is denoted by
\[
   \bH^{-3/2,-1/2}(\partial T) := \traceDD{T}(H(\div\Div\!,T)),\quad T\in\cT.
\]
These traces are supported on the boundary of the respective element since
\[
   \vdual{\div\Div\TTheta}{z}_T - \vdual{\TTheta}{\Grad\grad z}_T = 0\quad
   \forall \TTheta\in H(\div\Div\!,T),\ \forall z\in H^2_0(T),
\]
cf.~Proposition~\ref{prop_dd_jump} below.
It is therefore clear that, for given $\TTheta$,
the duality $\dual{\traceDD{T}(\TTheta)}{z}_{\partial T}$ only depends on the traces of $z$ and $\grad z$
on $\partial T$. Analogously, $\traceDD{T}(\TTheta)=0$ for any $\TTheta\in\DDs(T)$
(smooth symmetric tensors with support in $T$) since
\[
   \vdual{\div\Div\TTheta}{z}_T - \vdual{\TTheta}{\Grad\grad z}_T = 0\quad
   \forall \TTheta\in\DDs(T),\ \forall z\in H^2(T).
\]

\begin{remark} \label{rem_traceDD}
Let us define $\bH(\Div\!,T)$ as the closure of $\DDs(\overline{T})$ with respect to the norm
$(\|\cdot\|^2+\|\Div(\cdot)\|^2)^{1/2}$. It is clear that $\bH(\Div\!,T)$ is not a subspace of
$H(\div\Div\!,T)$, nor is $H(\div\Div\!,T)$ a subspace of $\bH(\Div\!,T)$
(see the second example in \S\ref{sec_num_ex}). This is precisely the
reason we have to consider the trace operator $\traceDD{T}$ in the form \eqref{trT_dd}.
When restricting this operator as
\begin{align*}
   \traceDD{T}:\;
   \left\{\begin{array}{cll}
      H(\div\Div\!,T)\cap \bH(\Div\!,T) & \to & H^2(T)',\\
      \TTheta & \mapsto & 
      \dual{\nn\cdot\Div\TTheta}{z}_{\partial T} - \dual{\TTheta\nn}{\grad z}_{\partial T},
   \end{array}\right.
\end{align*}
it reduces to standard trace operations.
In this case the two dualities are defined independently in the standard way,
\begin{align*}
   \dual{\nn\cdot\Div\TTheta}{z}_{\partial T} &:= \vdual{\Div\TTheta}{\grad z}_T + \vdual{\div\Div\TTheta}{z}_T,\\
   \dual{\TTheta\nn}{\grad z}_{\partial T} &:= \vdual{\TTheta}{\Grad\grad z}_T + \vdual{\Div\TTheta}{\grad z}_T.
\end{align*}
Now, in the three-dimensional case $\di=3$, defining the tangential trace
$\pi_\bt(\bphi):=\nn\times(\bphi\times\nn)|_{\partial T}$
for $\bphi\in\bD(\overline{T})$ and the surface gradient
$\grad_{\partial T}(\cdot):=\pi_\bt(\grad \cdot)|_{\partial T}$, we formally write
\begin{align} \label{surface}
   \dual{\TTheta\nn}{\grad z}_{\partial T}
   =
   \dual{\pi_\bt(\TTheta\nn)}{\grad_{\partial T} z}_{\partial T}
   + \dual{\nn\cdot\TTheta\nn}{\nn\cdot\grad z}_{\partial T}.
\end{align}
Correspondingly, in two dimensions ($\di=2$), we introduce the unit tangential vector $\bt$ along $\partial T$
in mathematically positive orientation, and use the notation
$\pi_\bt(\bphi):=(\bt\cdot\bphi)\bt|_{\partial T}$ for $\bphi\in\bD(\overline{T})$
with corresponding tangential derivative
$\grad_{\partial T}(\cdot):=\pi_\bt(\grad \cdot)|_{\partial T}$.
%=\bt\cdot\grad(\cdot)|_{\partial T}$.
Then, \eqref{surface} applies as well.
We also need the surface divergence $\divwo_{\partial T}(\cdot)$ defined by
$\dual{\divwo_{\partial T}(\bphi)}{z}_{\partial T}:=-\dual{\bphi}{\grad_{\partial T}z}_{\partial T}$
for sufficiently smooth vector functions $\bphi$ with $\pi_\bt(\bphi)=\bphi$.
For precise definitions and appropriate spaces we refer to \cite{BuffaCS_02_THL}.
With these definitions it is clear that we can define separate traces
\begin{align} \label{trace_n}
   \traceDD{T,\nn}:\;
   \left\{\begin{array}{cll}
      H(\div\Div\!,T) & \to & \bigl(H^2(T)\cap H^1_0(T)\bigr)'\\
      \TTheta & \mapsto & \dual{\nn\cdot\TTheta\nn}{\nn\cdot\grad z}_{\partial T} := -\dual{\traceDD{T}(\TTheta)}{z}_{\partial T}
   \end{array}\right.
\end{align}
and
\begin{align} \label{trace_t}
   \traceDD{T,\bt}:\;
   \left\{\begin{array}{cll}
      H(\div\Div\!,T) & \to & \{z\in H^2(T);\; \nn\cdot\grad z = 0\ \text{on}\ \partial T\}'\\
      \TTheta & \mapsto &
      \dual{\nn\cdot\Div\TTheta + \divwo_{\partial T}\pi_\bt(\TTheta\nn)}{z}_{\partial T}
      := \dual{\traceDD{T}(\TTheta)}{z}_{\partial T}
   \end{array}\right.
\end{align}
that coincide with the corresponding trace terms for sufficiently smooth functions $\TTheta$,
cf.~the operators $\gamma_0$ and $\gamma_1$ in \cite[page~1635]{AmaraCPC_02_BMM}.
On the one hand, these traces are relevant to identify basis functions for the approximation of traces of
$H(\div\Div\!,\Omega)$-functions.
On the other hand, specifying one of these traces, the other is well defined as a functional
acting on traces of $H^2$-functions (without the trace conditions for
$z$ in \eqref{trace_n} and \eqref{trace_t}).
Applying this on the boundary $\Gamma$ of $\Omega$, it is possible to specify
any physically meaningful boundary condition based on the terms
$\nn\cdot\MM\nn$, $\nn\cdot\Div\MM + \divwo_\Gamma\pi_\bt(\MM\nn)$, $u$, and $\nn\cdot\grad u$ on $\Gamma$.
Note that $\divwo_\Gamma$ refers to the operator that is dual to the (negative) global surface 
gradient $-\grad_\Gamma$, and integrating by parts piecewise on subsets of $\Gamma$ generates
a piecewise surface divergence plus jumps at the interfaces, cf.~\cite{RafetsederZ_DRK}.
Indeed, these jumps will be essential for the approximation analysis,
based on Proposition~\ref{prop_cor_dd_jump} below.
\end{remark}

The collective trace operator is defined by
\[
   \traceDD{}:\;
   \left\{\begin{array}{cll}
      H(\div\Div\!,\Omega) & \to & H^2(\cT)',\\
      \TTheta & \mapsto & \traceDD{}(\TTheta) := (\traceDD{T}(\TTheta))_T
   \end{array}\right.
\]
with duality
\begin{align} \label{tr_dd}
    \dual{\traceDD{}(\TTheta)}{z}_\cS := \sum_{T\in\cT} \dual{\traceDD{T}(\TTheta)}{z}_{\partial T}
\end{align}
and range
\begin{align*}
   \bH^{-3/2,-1/2}(\cS) := \traceDD{}(H(\div\Div\!,\Omega)).
\end{align*}
(Here, and in the following, considering dualities $\dual{\cdot}{\cdot}_{\partial T}$ on the whole of
$\partial T$, possibly involved traces onto $\partial T$ are always taken from $T$ without further notice.)
These global and local traces are measured in the \emph{minimum energy extension} norms,
\begin{align*}
   \|\bq\|_\trddiv{\cS} &:= \inf \Bigl\{\|\TTheta\|_{\div\Div};\; \TTheta\in H(\div\Div\!,\Omega),\ \traceDD{}(\TTheta)=\bq\Bigr\},\\
   \|\bq\|_\trddiv{\partial T} &:= \inf \Bigl\{\|\TTheta\|_{\div\Div\!,T};\; \TTheta\in H(\div\Div\!,T),\ \traceDD{T}(\TTheta)=\bq\Bigr\}.
\end{align*}
Alternative norms are defined by duality as follows,
\begin{align}
   \|\bq\|_{-3/2,-1/2,\partial T} &:=  \sup_{0\not=z\in H^2(T)} \frac{\dual{\bq}{z}_{\partial T}}{\|z\|_{2,T}},
   \quad \bq\in \bH^{-3/2,-1/2}(\partial T),\ T\in\cT,\nonumber\\
   \label{norm_tr_dd}
   \|\bq\|_{-3/2,-1/2,\cS} &:=  \sup_{0\not=z\in H^2(\cT)} \frac{\dual{\bq}{z}_\cS}{\|z\|_{2,\cT}},
   \quad \bq\in \bH^{-3/2,-1/2}(\cS).
\end{align}
Here, for given $\bq\in \bH^{-3/2,-1/2}(\partial T)$, the duality with $z\in H^2(T)$ is defined by
\[
   \dual{\bq}{z}_{\partial T} := \dual{\traceDD{T}(\TTheta)}{z}_{\partial T}
   \quad\text{for } \TTheta\in H(\div\Div\!,T)\text{ with } \traceDD{T}(\TTheta) = \bq,
\]
and, for $\bq=(\bq_T)_T\in \bH^{-3/2,-1/2}(\cS)$ and $z\in H^2(\cT)$,
\begin{align} \label{tr_dd_dual}
   \dual{\bq}{z}_\cS := \sum_{T\in\cT} \dual{\bq_T}{z}_{\partial T}.
\end{align}
This is consistent with definitions \eqref{trT_dd} and \eqref{tr_dd}.

\begin{lemma} \label{la_tr_dd_norms}
It holds the identity
\[
   \|\bq\|_{-3/2,-1/2,\partial T} = \|\bq\|_\trddiv{\partial T}\quad
   \forall \bq\in \bH^{-3/2,-1/2}(\partial T),\ T\in \cT,
\]
so that
\[
   \traceDD{T}:\; H(\div\Div\!,T)\to \bH^{-3/2,-1/2}(\partial T)
\]
has unit norm and $\bH^{-3/2,-1/2}(\partial T)$ is closed.
\end{lemma}

\begin{proof}
The estimate $\|\bq\|_{-3/2,-1/2,\partial T}\le \|\bq\|_\trddiv{\partial T}$ is immediate by bounding
\begin{align*}
   \dual{\traceDD{T}(\TTheta)}{z}_{\partial T} &\le \|\TTheta\|_{\div\Div\!,T} \|z\|_{2,T}
   \quad\forall\TTheta\in H(\div\Div\!,T),\ \forall z\in H^2(T),\ T\in\cT.
\end{align*}
Now let $T\in\cT$ and $\bq\in\bH^{-3/2,-1/2}(\partial T)$ be given.
We define $z\in H^2(T)$ as the solution to the problem
\begin{align} \label{prob_dd_z}
   \vdual{\Grad\grad z}{\Grad\grad \deltaz}_T + \vdual{z}{\deltaz}_T
   =
   \dual{\bq}{\deltaz}_{\partial T}
   \quad\forall \deltaz\in H^2(T).
\end{align}
Note that the right-hand side functional implies a natural boundary condition for $z$. Furthermore,
since $\dual{\bq}{\deltaz}_{\partial T}=0$ for $\deltaz\in H^2_0(T)$, $z$ satisfies
\begin{align} \label{pde_dd_z}
   \div\Div\Grad\grad z + z = 0\quad\text{in}\ T,
\end{align}
first in the distributional sense and, by the regularity of $z$, also in $L_2(T)$.
Using the function $z$ we continue to define $\TTheta\in H(\div\Div\!,T)$ as the solution to
\begin{align} \label{prob_dd_QQ}
   \vdual{\div\Div\TTheta}{\div\Div\Deltaq}_T + \vdual{\TTheta}{\Deltaq}_T
   =
   \dual{\traceDD{T}(\Deltaq)}{z}_{\partial T}
   \quad\forall\Deltaq\in H(\div\Div\!,T).
\end{align}
Again, the right-hand side functional induces a natural boundary condition for $\TTheta$, and it holds
\begin{align} \label{pde_dd_QQ}
   \Grad\grad\div\Div\TTheta + \TTheta = 0\quad\text{in}\ \LL_2^s(T).
\end{align}
We show that $\TTheta=-\Grad\grad z$. Indeed, defining $\TTheta^z:=-\Grad\grad z$, we find with \eqref{pde_dd_z} that
\(
   \div\Div\TTheta^z = z
\)
so that by definition of $\TTheta^z$ and definition \eqref{trT_dd} of $\traceDD{T}$, 
\begin{align*}
   \vdual{\div\Div\TTheta^z}{\div\Div\Deltaq}_T + \vdual{\TTheta^z}{\Deltaq}_T
   =
   \vdual{z}{\div\Div\Deltaq}_T - \vdual{\Grad\grad z}{\Deltaq}_T
   =
   \dual{\traceDD{T}(\Deltaq)}{z}_{\partial T}
\end{align*}
for any $\Deltaq\in H(\div\Div\!,T)$. This shows that $\TTheta^z$ solves \eqref{prob_dd_QQ}
and by uniqueness, $\TTheta=\TTheta^z=-\Grad\grad z$.
Using this relation and $\div\Div\TTheta^z = z$, it follows by \eqref{prob_dd_z} that
\begin{align*}
   \dual{\traceDD{T}(\TTheta)}{\deltaz}_{\partial T}
   &=
   \vdual{\div\Div\TTheta}{\deltaz}_T - \vdual{\TTheta}{\Grad\grad\deltaz}_T
   \\
   &=
   \vdual{z}{\deltaz}_T + \vdual{\Grad\grad z}{\Grad\grad\deltaz}_T
   =
   \dual{\bq}{\deltaz}_{\partial T}
   \quad\forall\deltaz\in H^2(T).
\end{align*}
In other words, $\traceDD{T}(\TTheta)=\bq$.
This relation together with selecting $\deltaz=z$ in \eqref{prob_dd_z} and $\Deltaq=\TTheta$ in \eqref{prob_dd_QQ}, shows that
\begin{align} \label{pf_tr_dd_norms_1}
   \dual{\bq}{z}_{\partial T} = \|z\|_{2,T}^2 = \dual{\traceDD{T}(\TTheta)}{z}_{\partial T} = \|\TTheta\|_{\div\Div\!,T}^2.
\end{align}
Noting that
\[
   \|\TTheta\|_{\div\Div\!,T} = \inf\{\|\wilde\TTheta\|_{\div\Div\!,T};\; \wilde\TTheta\in H(\div\Div\!,T),\ \traceDD{T}(\wilde\TTheta)=\bq\}
                        = \|\bq\|_\trddiv{\partial T}
\]
by \eqref{pde_dd_QQ}, relation \eqref{pf_tr_dd_norms_1} finishes the proof of the norm identity.
The space $\bH^{-3/2,-1/2}(\partial T)$ is closed as the image of a bounded below operator.
\end{proof}

%===================================================================================================
\subsection{Traces of $H^2_0(\Omega)$} \label{sec_trace_gg}

Let us study traces of $H^2_0(\Omega)$ in a similar way as $H(\div\Div\!,\Omega)$ in the previous section.

We define linear operators $\traceGG{T}:\;H^2(T)\to H(\div\Div\!,T)'$ for $T\in\cT$ by
\begin{align} \label{trT_gg}
   \dual{\traceGG{T}(z)}{\TTheta}_{\partial T} := \vdual{\div\Div\TTheta}{z}_T - \vdual{\TTheta}{\Grad\grad z}_T
\end{align}
and observe that (cf.~\eqref{trT_dd})
\begin{equation} \label{traces_id}
   \dual{\traceGG{T}(z)}{\TTheta}_{\partial T} = \dual{\traceDD{T}(\TTheta)}{z}_{\partial T}
   \quad\forall z\in H^2(T),\ \forall\TTheta\in H(\div\Div\!,T).
\end{equation}
The ranges are denoted by
\[
   \bH^{3/2,1/2}(\partial T) := \traceGG{T}(H^2(T)),\quad T\in\cT.
\]
It is immediate that $\traceGG{T}(z)=0$ if and only if $z\in H^2_0(T)$.
The collective trace operator is defined by
\[
   \traceGG{}:\;
   \left\{\begin{array}{cll}
      H^2_0(\Omega) & \to & H(\div\Div\!,\cT)',\\
      z & \mapsto & \traceGG{}(z) := (\traceGG{T}(z))_T
   \end{array}\right.
\]
with duality
\begin{align} \label{tr_gg}
    \dual{\traceGG{}(z)}{\TTheta}_\cS := \sum_{T\in\cT} \dual{\traceGG{T}(z)}{\TTheta}_{\partial T}
\end{align}
and range
\begin{align*}
   \bH^{3/2,1/2}_{00}(\cS) := \traceGG{}(H^2_0(\Omega)).
\end{align*}
These trace spaces are provided with canonical trace norms,
\begin{align*}
   &\|\bv\|_\trggrad{\partial T} = \inf\{\|v\|_{2,T};\; v\in H^2(T),\ \traceGG{T}(v)=\bv\},\quad T\in\cT,\\
   &\|\bv\|_\trggrad{0,\cS} = \inf\{\|v\|_2;\; v\in H^2_0(\Omega),\ \traceGG{}(v)=\bv\}.
\end{align*}
Alternative norms are defined by duality as follows,
\begin{align*}
   \|\bv\|_{3/2,1/2,\partial T} &:= 
   \sup_{0\not=\TTheta\in H(\div\Div\!,T)} \frac{\dual{\bv}{\TTheta}_{\partial T}}{\|\TTheta\|_{\div\Div\!,T}},
   \quad \bv\in \bH^{3/2,1/2}(\partial T),\ T\in\cT,\nonumber\\
   %\label{norm_tr_gg}
   \|\bv\|_{3/2,1/2,00,\cS} &:=
   \sup_{0\not=\TTheta\in H(\div\Div\!,\cT)} \frac{\dual{\bv}{\TTheta}_\cS}{\|\TTheta\|_{\div\Div\!,\cT}},
   \quad \bv\in \bH^{3/2,1/2}_{00}(\cS).
\end{align*}
Here, for given $\bv\in \bH^{3/2,1/2}(\partial T)$, the duality with $\TTheta\in H(\div\Div\!,T)$ is defined by
\[
   \dual{\bv}{\TTheta}_{\partial T} := \dual{\traceGG{T}(z)}{\TTheta}_{\partial T}
   \quad\text{for } z\in H^2(T)\text{ with } \traceGG{T}(z) = \bv,
\]
and, for $\bv=(\bv_T)_T\in \bH^{3/2,1/2}_{00}(\cS)$ and $\TTheta\in H(\div\Div\!,\cT)$,
\begin{align} \label{tr_gg_dual}
   \dual{\bv}{\TTheta}_\cS := \dual{\traceGG{}(z)}{\TTheta}_\cS
   \quad\text{for } z\in H^2_0(\Omega)\text{ with } \traceGG{}(z) = \bv.
\end{align}
This is consistent with definitions \eqref{trT_gg} and \eqref{tr_gg}.

\begin{lemma} \label{la_tr_gg_norms}
It holds the identity
\[
   \|\bv\|_{3/2,1/2,\partial T} = \|\bv\|_\trggrad{\partial T}\quad
   \forall \bv\in \bH^{3/2,1/2}(\partial T),\ T\in \cT,
\]
so that
\[
   \traceGG{T}:\; H^2(T)\to \bH^{3/2,1/2}(\partial T)
\]
has unit norm and $\bH^{3/2,1/2}(\partial T)$ is closed.
\end{lemma}

\begin{proof}
The proof is very similar to that of Lemma~\ref{la_tr_dd_norms}.

The estimate $\|\bv\|_{3/2,1/2,\partial T}\le \|\bv\|_\trggrad{\partial T}$ follows by bounding
\begin{align*}
   \dual{\traceGG{T}(z)}{\TTheta}_{\partial T} &\le \|z\|_{2,T} \|\TTheta\|_{\div\Div\!,T}
   \quad\forall z\in H^2(T),\ \forall \TTheta\in H(\div\Div\!,T),\ T\in\cT.
\end{align*}
To show the other inequality, let $T\in\cT$ and $\bv\in\bH^{3/2,1/2}(\partial T)$ be given.
We define $\TTheta\in H(\div\Div\!,T)$ as the solution of
\begin{align} \label{prob_gg_QQ}
   \vdual{\div\Div\TTheta}{\div\Div\Deltaq}_T + \vdual{\TTheta}{\Deltaq}_T
   =
   \dual{\bv}{\Deltaq}_{\partial T}
   \quad\forall \Deltaq\in H(\div\Div\!,T).
\end{align}
It satisfies
\begin{align} \label{pde_gg_QQ}
   \Grad\grad\div\Div \TTheta + \TTheta = 0\quad\text{in}\ \LL_2^s(T).
\end{align}
We continue to define $z\in H^2(T)$ by
\begin{align} \label{prob_gg_z}
   \vdual{\Grad\grad z}{\Grad\grad\deltaz}_T + \vdual{z}{\deltaz}_T
   =
   \dual{\traceGG{T}(\deltaz)}{\TTheta}_{\partial T}
   \quad\forall\deltaz\in H^2(T).
\end{align}
It satisfies
\begin{align} \label{pde_gg_z}
   \div\Div\Grad\grad z + z = 0\quad\text{in}\ L_2(T),
\end{align}
and we conclude that $z=\div\Div\TTheta$ as follows.
Defining $z^\TTheta:=\div\Div\TTheta$, \eqref{pde_gg_QQ} shows that
\(
   \Grad\grad z^\TTheta = -\TTheta
\)
so that by definition of $z^\TTheta$ and $\traceGG{T}$ (cf.~\eqref{trT_gg}), 
\begin{align*}
   \vdual{\Grad\grad z^\TTheta}{\Grad\grad\deltaz}_T + \vdual{z^\TTheta}{\deltaz}_T
   =
   -\vdual{\TTheta}{\Grad\grad\deltaz}_T + \vdual{\div\Div \TTheta}{\deltaz}_T
   =
   \dual{\traceGG{T}(\deltaz)}{\TTheta}_{\partial T}
\end{align*}
for any $\deltaz\in H^2(T)$. Hence $z^\TTheta=z$ solves \eqref{prob_gg_z}, that is,
$z=\div\Div\TTheta$.
Using this relation and $\Grad\grad z = -\TTheta$, \eqref{prob_gg_QQ} shows that
\begin{align*}
   \dual{\traceGG{T}(z)}{\Deltaq}_{\partial T}
   &=
   \vdual{\div\Div\Deltaq}{z}_T - \vdual{\Deltaq}{\Grad\grad z}_T
   \\
   &=
   \vdual{\div\Div\Deltaq}{\div\Div\TTheta}_T + \vdual{\Deltaq}{\TTheta}_T
   =
   \dual{\bv}{\Deltaq}_{\partial T}
   \quad\forall\Deltaq\in H(\div\Div\!,T),
\end{align*}
so that $\traceGG{T}(z)=\bv$.
Then selecting $\Deltaq=\TTheta$ in \eqref{prob_gg_QQ} and $\deltaz=z$ in \eqref{prob_gg_z}, we obtain
\begin{align} \label{pf_tr_gg_norms_1}
   \|\TTheta\|_{\div\Div\!,T}^2 = \dual{\bv}{\TTheta}_{\partial T} = \|z\|_{2,T}^2.
\end{align}
Since
\[
   \|z\|_{2,T} = \inf\{\|\wilde z\|_{2,T};\; \wilde z\in H^2(T),\ \traceGG{T}(\wilde z)=\bv\}
                        = \|\bv\|_\trggrad{\partial T}
\]
by \eqref{pde_gg_z}, relation \eqref{pf_tr_gg_norms_1} finishes the proof of the norm identity.
The space $\bH^{3/2,1/2}(\partial T)$ is closed as the image of a bounded below operator.
\end{proof}

%===================================================================================================
\subsection{Jumps of $H(\div\Div\!,\cT)$} \label{sec_jump_dd}

\begin{prop} \label{prop_dd_jump}
(i) For $\TTheta\in H(\div\Div\!,\cT)$ it holds
\[
   \TTheta\in H(\div\Div\!,\Omega) \quad\Leftrightarrow\quad
   \dual{\traceGG{}(z)}{\TTheta}_\cS = 0\quad\forall z\in H^2_0(\Omega).
\]
(ii) The identity
\begin{align*}
   \sum_{T\in\cT} \|\bq\|_\trddiv{\partial T}^2 = \|\bq\|_\trddiv{\cS}^2
   \quad\forall \bq\in \bH^{-3/2,-1/2}(\cS)
\end{align*}
holds true.
\end{prop}

\begin{proof}
The proof of (i) follows the standard procedure, cf.~\cite[Proof of Theorem 2.3]{CarstensenDG_16_BSF}.
For $\TTheta\in H(\div\Div\!,\Omega)$ and $z\in\mathcal{D}(\Omega)$,
\begin{align*}
   \dual{\traceGG{}(z)}{\TTheta}_\cS
   &\overset{\mathrm{def}}=
   \sum_{T\in\cT} \vdual{\div\Div\TTheta}{z}_T - \vdual{\TTheta}{\Grad\grad z}_T
   =
   \vdual{\div\Div\TTheta}{z} - \vdual{\TTheta}{\Grad\grad z} = 0,
\end{align*}
showing the direction ``$\Rightarrow$''.
Now, for given $\TTheta\in H(\div\Div\!,\cT)$ with $\dual{\traceGG{}(z)}{\TTheta}_\cS = 0$ for any $z\in H^2_0(\Omega)$
we have in the distributional sense
\[
   \div\Div\TTheta(z)=\vdual{\TTheta}{\Grad\grad z}=\vdual{\div\Div\TTheta}{z}_\cT - \dual{\traceGG{}(z)}{\TTheta}_\cS
   = \vdual{\div\Div\TTheta}{z}_\cT\quad\forall z\in\mathcal{D}(\Omega).
\]
Therefore, $\div\Div\TTheta\in L_2(\Omega)$, that is, $\TTheta\in H(\div\Div\!,\Omega)$.

Next we show (ii).
The inequality $\sum_{T\in\cT} \|\bq\|_\trddiv{\partial T}^2 \le \|\bq\|_\trddiv{\cS}^2$
holds by definition of the norms. To show the other inequality let
$\bq=(\bq_T)_T\in \bH^{-3/2,-1/2}(\cS)$ be given. By definition of $\bH^{-3/2,-1/2}(\cS)$
there is $\TTheta\in H(\div\Div\!,\Omega)$ such that $\traceDD{}(\TTheta)=\bq$.
Furthermore, for $T\in\cT$, let $\wilde\TTheta_T\in H(\div\Div\!,T)$
be such that $\traceDD{T}(\wilde\TTheta_T)=\bq_T$ and
\[
   \|\bq_T\|_\trddiv{\partial T}
   =
   \|\wilde\TTheta_T\|_{\div\Div\!,T}.
\]
Then, $\wilde\TTheta\in H(\div\Div\!,\cT)$ defined by $\wilde\TTheta|_T:=\wilde\TTheta_T$ ($T\in\cT$) satisfies
(cf.~\eqref{traces_id})
\begin{align*}
   \dual{\traceGG{}(z)}{\wilde\TTheta}_\cS
   &=
   \sum_{T\in\cT} \dual{\traceGG{T}(z)}{\wilde\TTheta_T}_{\partial T}
   =
   \sum_{T\in\cT} \dual{\traceDD{T}(\wilde\TTheta_T)}{z}_{\partial T}
   \\
   &=
   \sum_{T\in\cT}\dual{\traceDD{T}(\TTheta)}{z}_{\partial T}
   =
   \dual{\traceGG{}(z)}{\TTheta}_\cS
   =0\quad\forall z\in H^2_0(\Omega)
\end{align*}
by part (i). Again with (i) we conclude that $\wilde\TTheta\in H(\div\Div\!,\Omega)$.
Therefore,
\begin{align*}
   \sum_{T\in\cT} \|\bq\|_\trddiv{\partial T}^2
   &=
   \sum_{T\in\cT} \|\wilde\TTheta_T\|_{\div\Div\!,T}^2 = \|\wilde\TTheta\|_{\div\Div}^2
   \ge
   \|\bq\|_\trddiv{\cS}^2.
\end{align*}
This finishes the proof.
\end{proof}

By Proposition~\ref{prop_dd_jump}, for given $\bv\in \bH^{3/2,1/2}_{00}(\cS)$,
$\dual{\bv}{\TTheta}_\cS$ defines a functional that only depends on the normal jumps of $\TTheta$ and $\pwDiv\TTheta$
across the element interfaces. It will be denoted as
\begin{align} \label{jumpQQ}
   \jump{(\cdot)\nn,\nn\cdot\pwDiv(\cdot)}:\;
   \left\{\begin{array}{cll}
      H(\div\Div\!,\cT) & \to & \Bigl(\bH^{3/2,1/2}_{00}(\cS)\Bigr)'\\
      \TTheta        & \mapsto & \jump{\TTheta\nn,\nn\cdot\pwDiv\TTheta}(\bv) := \dual{\bv}{\TTheta}_\cS
   \end{array}\right.
\end{align}
with duality pairing defined in \eqref{tr_gg_dual}. This functional defines a semi-norm in $H(\div\Div\!,\cT)$,
\begin{align} \label{seminorm_QQ}
   \|\jump{\TTheta\nn,\nn\cdot\pwDiv\TTheta}\|_{(3/2,1/2,00,\cS)'}
   &:=
   \sup_{0\not=\bv\in \bH^{3/2,1/2}_{00}(\cS)}
   \frac{\dual{\bv}{\TTheta}_\cS}{\|\bv\|_{3/2,1/2,00,\cS}},
   \quad \TTheta\in H(\div\Div\!,\cT).
\end{align}

\begin{prop} \label{prop_dd_trace}
It holds the identity
\[
   \|\bq\|_{-3/2,-1/2,\cS} = \|\bq\|_\trddiv{\cS}
   \quad\forall\bq\in\bH^{-3/2,-1/2}(\cS).
\]
In particular,
\[
   \traceDD{}:\; H(\div\Div\!,\Omega)\to \bH^{-3/2,-1/2}(\cS)
\]
has unit norm and $\bH^{-3/2,-1/2}(\cS)$ is closed.
\end{prop}

\begin{proof}
The norm identity is shown by standard duality arguments in product spaces, 
cf.~\cite[Theorem~2.3]{CarstensenDG_16_BSF}.
Specifically, for $\bq=(\bq_T)_T\in H^{-3/2,-1/2}(\cS)$ we calculate,
by using Proposition~\eqref{prop_dd_jump}(ii) and Lemma~\ref{la_tr_dd_norms},
\begin{align*}
   \|\bq\|_{-3/2,-1/2,\cS}^2
   &=
   \Bigl(\sup_{0\not=z\in H^2(\cT)} \frac{\sum_{T\in\cT}\dual{\bq_T}{z}_{\partial T}}{\|z\|_{2,\cT}}\Bigr)^2
   =
   \sum_{T\in\cT}
   \sup_{0\not=z\in H^2(T)} \frac{\dual{\bq_T}{z}_{\partial T}^2}{\|z\|_{2,T}^2}
   \\
   &=
   \sum_{T\in\cT} \|\bq_T\|_{-3/2,-1/2,\partial T}^2
   =
   \sum_{T\in\cT} \|\bq_T\|_\trddiv{\partial T}^2
   =
   \|\bq\|_\trddiv{\cS}^2.
\end{align*}
% Old proof of closedness:
% By the continuity of $\traceDD{}$, its kernel is closed in $H(\div\Div\!,\Omega)$.
% Therefore, the quotient space $H(\div\Div\!,\Omega)/\ker(\traceDD{})$ is closed with norm
% $\|\cdot\|_{H(\div\Div\!,\Omega)/\ker(\traceDD{})}$,
% which is equal to the minimum energy extension norm $\|\cdot\|_\trddiv{\cS}$.
% The previously shown norm identity implies that
% \[
%    \|\traceDD{T}(\TTheta)\|_{-3/2,-1/2,\cS} = \|\TTheta\|_{H(\div\Div\!,\Omega)/\ker(\traceDD{})}
%    \quad\forall\TTheta\in H(\div\Div\!,\Omega).
% \]
% Thus, the trace operator is bounded below and its range is closed.
The space $\bH^{-3/2,-1/2}(\cS)$ is closed as the image of a bounded below operator.
\end{proof}

%===================================================================================================
\subsection{Traces of $H(\div\Div\!,\cT)$} \label{sec_jump2_dd}

For the discretization of $\traceDD{}(H(\div\Div\!,\Omega))$ we need a characterization of continuity
across the skeleton interfaces $\partial T\in\cS$ that is based on local traces,
rather than testing with $H^2_0(\Omega)$-functions as in Proposition~\ref{prop_dd_jump}.
Therefore, in this section, we assume throughout that the mesh $\cT$ consists of polyhedra
($\di=3$) or polygons ($\di=2$) with possibly curved faces/edges, and that $\cT$ is a conforming
subdivision of $\Omega$ in the sense that the intersection of any two different (closed) elements
is either empty, an entire face ($\di=3$), an entire edge ($\di=2,3$), or a vertex ($\di=2,3$)
of both elements.

\begin{figure}[htb]
  \begin{center}
    \includegraphics{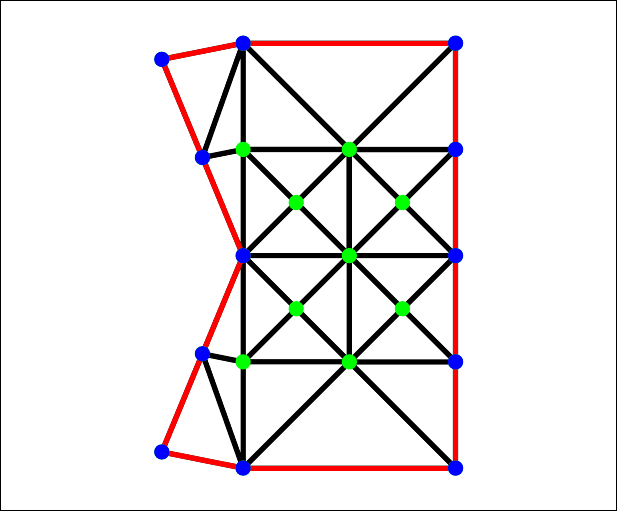}
    \includegraphics{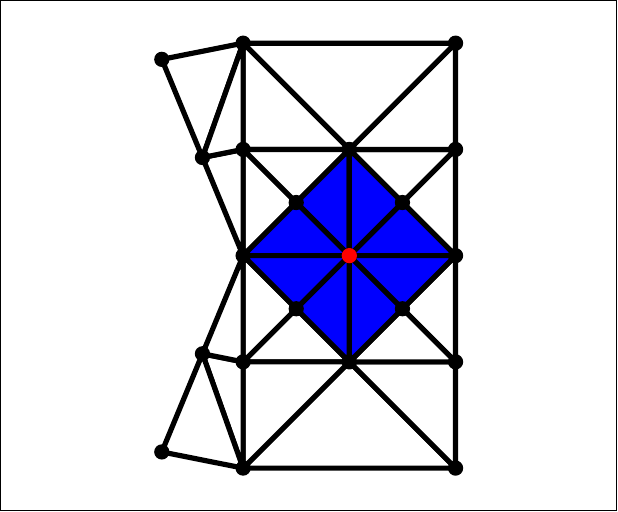}
  \end{center}
  \caption{Notational convention for $d=2$.
  Left: The dots visualize the set of all vertices $\cN$. The dots highlighted in blue indicate boundary vertices.
  Consequently, green dots visualize the set of all interior vertices $\cN_0$.
  Lines between two dots visualize the set of edges $\cE$. Similarly, red lines indicate boundary edges,
  whereas black lines correspond to the set of interior edges $\cE_0$.
  Right: The shaded elements (blue) indicate the patch $\omega(e)$ of an interior node $e$
  that is highlighted (red).}
  \label{fig:NodesEdges}
\end{figure} 

Let us introduce the set $\cE_T$ of faces ($\di=3$) or edges ($\di=2$)
of $T\in\cT$, and define $\cE_0$ to be the set of all faces/edges of $\cT$ that are not subsets of
$\Gamma$. We also need the set $\cN_E$ of edges ($\di=3$) or nodes ($\di=2$) of $E\in\cE_T$, $T\in\cT$, and
the set $\cN_T=\cup_{E\in\cE_T}\cN_E$ of all edges/nodes of an element $T\in\cT$.
The set of all edges/nodes $e\in\cN:=\cup_{T\in\cT}\cN_T$ that are not subsets of
$\Gamma$ is denoted by $\cN_0$, cf. the left side of Figure~\ref{fig:NodesEdges}.
For each $e\in\cN_0$, let $\omega(e)\subset\cT$ be the set (patch) of
elements $T\in\cT$ with $e\subset\overline{T}$, cf. the right side of Figure~\ref{fig:NodesEdges}.
The domain generated by a patch $\omega(e)$ will be denoted by $\omega_e$.
In three space dimensions, for a face $E\in\cE_T$ ($T\in\cT$),
$\nn_E$ denotes the unit normal vector along $\partial E$ that is tangential to $E$.
For an edge $E\in\cE_T$ ($\di=2$), $\nn_E$ indicates the orientation
of $E$, with values $\nn_E(e_1)=-1$ and $\nn_E(e_2)=1$, $e_1,e_2\in\cN_E$ being the starting
and end points of $E$.

We also need the following trace spaces of $H^2(T)$ for $E\in\cE_T$, $T\in\cT$,
%$e\in\cN_E$,
\begin{align*}
   H^{3/2}(E) &:= \{z|_E;\; z\in H^2(T)\},\quad
%   H^1(e) = \{z|_e;\; z\in H^2(T)\},\\
   H^{1/2}(E) := \{(\nn\cdot\grad z)|_E;\; z\in H^2(T)\}
\end{align*}
with canonical trace norms.
%Note that for $\di=2$, $H^1(e)$ consists of point values at $e$.
Now, to localize the representation of the trace operators $\traceDD{T,\bt}$ (recall \eqref{trace_t}),
instead of the surface divergence operator $\divwo_{\partial T}$, we need the local surface divergence
operator $\divwo_E$ defined, for a sufficiently smooth tangential function $\bphi=\pi_\bt(\bphi)$, by
$\dual{\divwo_E\bphi}{\varphi}_E:=-\dual{\bphi}{\grad_{\partial T}\varphi}_E$
for any $\varphi\in H^1_0(E)$ (with obvious definition of this space).
Defining
\[
   H(\divwo_E,E) := \{\bphi\in\bL_2(E);\; \pi_\bt(\bphi)=\bphi,\; \divwo_E\bphi\in L_2(E)\}
   \quad (E\in\cE_T,\; T\in\cT),
\]
we then have that
\[
   \dual{\bphi}{\grad_{\partial T} z}_E
   =
   -\dual{\divwo_E\bphi}{z}_E + \dual{\nn_E\cdot\bphi}{z}_{\partial E}
   \quad\forall\bphi\in H(\divwo_E,E),\ z\in H^1(E),\ E\in\cE_T,\ T\in\cT.
\]
For an element $T\in\cT$ and a sufficiently smooth function $\TTheta\in H(\div\Div\!,T)$,
we introduce local trace operators (cf.~\eqref{trace_n} and \eqref{trace_t})
\begin{equation} \label{traces_tn}
   \traceDD{T,E,\bt}(\TTheta)
   :=
   \bigl(\nn\cdot\Div\TTheta+\divwo_E\pi_\bt(\TTheta\nn)\bigr)|_E,\quad
   \traceDD{T,E,\nn}(\TTheta)
   :=
   \nn\cdot\TTheta\nn|_E
   \quad (E\in\cE_T)
\end{equation}
and the jump functional
\begin{equation} \label{jjump}
   \jjump{\TTheta}_{\partial T}(z)
   :=
   \sum_{E\in\cE_T}
   \bigl(\dual{\traceDD{T,E,\bt}(\TTheta)}{z}_E - \dual{\traceDD{T,E,\nn}(\TTheta)}{\nn\cdot\grad z}_E\bigr)
   -
   \dual{\traceDD{T}(\TTheta)}{z}_{\partial T}
\end{equation}
for $z\in H^2(T)$.
Of course, for sufficiently smooth $\TTheta$ it holds $\jjump{\TTheta}_{\partial T}\in \bigl(H^2(T)\bigr)'$.
Below we will require that the regularity of $\TTheta$ is such that the traces \eqref{traces_tn} are well defined.
Assuming, again, sufficient regularity of $\TTheta$, integration by
parts shows that the jump functional reduces to
\begin{equation} \label{jjump_loc}
\begin{split}
   &\jjump{\TTheta}_{\partial T}(z)
   =
   \dual{\nn_{E_2}\cdot\pi_\bt(\TTheta\nn|_{E_2}) + \nn_{E_1}\cdot\pi_\bt(\TTheta\nn|_{E_1})}{z}_e,\\
   &\text{for}\ E_1,E_2\in\cE_T:\; E_1\not=E_2,\ e=\overline{E}_1\cap\overline{E}_2\quad
   \text{and}\ z\in H^2(T):\; z|_{\tilde e}=0\ \forall \tilde e\in\cN_E\setminus\{e\}.
\end{split}
\end{equation}
Here, in three dimensions, $\dual{\varphi}{\psi}_e$ is $L_2(e)$-bilinear form (and its extension
by duality) and in two dimensions, $\dual{\varphi}{\psi}_e=\varphi(e)\psi(e)$ is the product
of the point values of $\varphi$ and $\psi$ at the node $e$, and $z|_e=z(e)$.

In order to be able to localize the traces of a function
$\TTheta\in H(\div\Div\!,T)$, according to \eqref{traces_tn}, we need to assume the stronger regularity
$\TTheta\in\cH(\div\Div\!,T)$ where
\[
\begin{split}
   &\cH(\div\Div\!,T):=\\
   &\quad
   \{\TTheta\in H(\div\Div\!,T);\;
      \traceDD{T,E,\bt}(\TTheta)\in \bigl(H^{3/2}(E)\bigr)',\
      \traceDD{T,E,\nn}(\TTheta)\in \bigl(H^{1/2}(E)\bigr)'\ \forall E\in\cE_T\}.
\end{split}
\]
The corresponding product space is denoted by $\cH(\div\Div\!,\cT)$.
Since $\DDs(\overline{T})\subset\cH(\div\Div\!,T)\subset H(\div\Div\!,T)$, the space
$\cH(\div\Div\!,T)$ is dense in $H(\div\Div\!,T)$.

Now we can formulate the main result of this subsection.
\begin{prop} \label{prop_cor_dd_jump}
An element $\TTheta\in \cH(\div\Div\!,\cT)$ satisfies $\TTheta\in H(\div\Div\!,\Omega)$ if and only if
\begin{equation} \label{cont}
\begin{split}
   &\traceDD{T_1,E,\nn}(\TTheta) + \traceDD{T_2,E,\nn}(\TTheta) = 0,\quad
   \traceDD{T_1,E,\bt}(\TTheta) + \traceDD{T_2,E,\bt}(\TTheta) = 0\\
   &\qquad\qquad\forall E\in\cE_0\ \text{and}\ T_1,T_2\in\cT:\ T_1\not=T_2,\ \{E\}=\cE_{T_1}\cap\cE_{T_2}\\
   \text{and}\quad
   &\sum_{T\in\omega(e)}\jjump{\TTheta}_{\partial T}(z) = 0
   \quad\forall z\in H^2_0(\omega_e),\;\forall e\in\cN_0.
\end{split}
\end{equation}
\end{prop}

\begin{proof}
For sufficiently smooth $\TTheta_T\in H(\div\Div\!,T)$ and $z\in H^2(T)$ ($T\in\cT$) we find
with \eqref{traces_id} and \eqref{jjump} that
\begin{align} \label{pf_cor_prop}
\lefteqn{
   \dual{\traceGG{T}(z)}{\TTheta_T}_{\partial T}
   =
   \dual{\traceDD{T}(\TTheta_T)}{z}_{\partial T}
}
   \nonumber\\
   &=
   \sum_{E\in\cE_T}
   \bigl(
   \dual{\traceDD{T,E,\bt}(\TTheta_T)}{z}_E
   -
   \dual{\traceDD{T,E,\nn}(\TTheta_T)}{\nn\cdot\grad z}_E
   \bigr)
   -
   \jjump{\TTheta_T}_{\partial T}(z).
\end{align}
All terms can be interpreted as linear functionals depending on $z$ and acting on $\TTheta_T$.
Boundedness is guaranteed for $\TTheta_T\in\cH(\div\Div\!,T)$.
Therefore, relation \eqref{pf_cor_prop} extends by continuity to $\TTheta\in\cH(\div\Div\!,\cT)$.
Considering $z\in H^2_0(\Omega)$ and summing over all elements $T\in\cT$ yields
\begin{align*}
   \dual{\traceGG{}(z)}{\TTheta}_\cS
   =
   \sum_{T\in\cT}\sum_{E\in\cE_T}
   \bigl(
   \dual{\traceDD{T,E,\bt}(\TTheta_T)}{z}_E
   -
   \dual{\traceDD{T,E,\nn}(\TTheta_T)}{\nn\cdot\grad z}_E
   \bigr)
   -
   \sum_{T\in\cT} \jjump{\TTheta_T}_{\partial T}(z).
\end{align*}
One sees that the right-hand side vanishes for any $z\in H^2_0(\Omega)$ if and only if
\eqref{cont} is satisfied. Therefore, the statement follows by Proposition~\ref{prop_dd_jump}.
\end{proof}

\begin{remark} \label{rem:traceOp2D}
In two dimensions ($\di=2$) the trace operators \eqref{traces_tn} are, for $\TTheta\in\cH(\div\Div\!,T)$
and $T\in\cT$,
\begin{equation} \label{traces_tn_2d}
   \traceDD{T,E,\bt}(\TTheta)
   =
   \bigl(\nn\cdot\Div\TTheta+\partial_\bt(\bt\cdot\TTheta\nn)\bigr)|_E,\quad
   \traceDD{T,E,\nn}(\TTheta)
   =
   \nn\cdot\TTheta\nn|_E
   \quad (E\in\cE_T).
\end{equation}
Here, $\partial_\bt$ indicates the positively oriented tangential derivative along $E$.
The localized jump functional $\jjump{\TTheta}_{\partial T}$ (cf.~\eqref{jjump_loc})
reduces to jump values at vertices,
\begin{equation} \label{jjump_node}
\begin{split}
   \jjump{\TTheta}_{\partial T}(e):\;z\mapsto
   \jjump{\TTheta}_{\partial T}(z)
   &=
   \dual{\nn_{E_2}\cdot\pi_\bt(\TTheta\nn|_{E_2}) + \nn_{E_1}\cdot\pi_\bt(\TTheta\nn|_{E_1})}{z}_e\\
   &=
   \Bigl((\bt\cdot\TTheta\nn|_{E_2})(e)- (\bt\cdot\TTheta\nn|_{E_1})(e)\Bigr) z(e)
   \quad\forall z\in H^2_0(\omega_e)
\end{split}
\end{equation}
where $E_1,E_2\in\cE_T$ are chosen in such a way that $e$ is the endpoint of $E_2$ and starting point
of $E_1$, that is, in our previous notation, $\nn_{E_1}(e)=-1$, $\nn_{E_2}(e)=1$.
\end{remark}

%===================================================================================================
\subsection{Jumps of $H^2(\cT)$} \label{sec_jump_gg}

\begin{prop} \label{prop_gg_jump}
(i) For $z\in H^2(\cT)$ the following equivalence holds,
\[
   z\in H^2_0(\Omega)\quad\Leftrightarrow\quad
   \dual{\bq}{z}_\cS=0 \quad\forall\bq\in \bH^{-3/2,-1/2}(\cS).
\]
(ii) The identity
\begin{align*}
   \sum_{T\in\cT} \|\bv\|_\trggrad{\partial T}^2 = \|\bv\|_\trggrad{0,\cS}^2
   \quad\forall \bv\in \bH^{3/2,1/2}_{00}(\cS)
\end{align*}
holds true.
\end{prop}

\begin{proof}
Some parts of this proof are almost identical to the proof of Proposition~\ref{prop_dd_jump}.
However, in part (i) of the present case, we additionally have to guarantee boundary conditions
in $H^2_0(\Omega)$ whereas previously they were not considered.

Let us start showing statement (i).
For $z\in\mathcal{D}(\Omega)$, $\bq\in\bH^{-3/2,-1/2}(\cS)$ and $\TTheta\in H(\div\Div\!,\Omega)$
with $\bq=\traceDD{}(\TTheta)$ we obtain
\begin{align*}
   \dual{\bq}{z}_\cS
   &\overset{\mathrm{def}}=
   \sum_{T\in\cT} \vdual{\div\Div\TTheta}{z}_T - \vdual{\TTheta}{\Grad\grad z}_T
   \\
   &=
   \vdual{\div\Div\TTheta}{z} - \vdual{\TTheta}{\Grad\grad z}
   =
   \vdual{\TTheta}{\Grad\grad z} - \vdual{\TTheta}{\Grad\grad z} = 0.
\end{align*}
By density, this holds for any $z\in H^2_0(\Omega)$.

We show the other direction ``$\Leftarrow$''.
For given $z\in H^2(\cT)$ with $\dual{\bq}{z}_\cS=0$ for any $\bq\in \bH^{-3/2,-1/2}(\cS)$,
the distribution $\Grad\grad z$ satisfies
\begin{align*}
   \Grad\grad z(\PPhi) = \vdual{z}{\div\Div\PPhi}
   =
   \vdual{\Grad\grad z}{\PPhi}_\cT + \dual{\traceDD{}(\PPhi)}{z}_\cS
   = \vdual{\Grad\grad z}{\PPhi}_\cT\quad\forall\PPhi\in\DDs(\Omega),
\end{align*}
that is, $\Grad\grad z\in \LL_2^s(\Omega)$. Therefore, $z\in H^2(\Omega)$.
Furthermore,
\begin{align} \label{z_trace}
   0 = \dual{\traceDD{}(\PPhi)}{z}_\cS
   &=
   \vdual{\div\Div\PPhi}{z} - \vdual{\PPhi}{\pwGrad\pwgrad z}
   =
   \vdual{\div\Div\PPhi}{z} - \vdual{\PPhi}{\Grad\grad z}
\end{align}
for any $\PPhi\in H(\div\Div\!,\Omega)$ shows that $z\in H^2_0(\Omega)$, as can be seen as follows.
To show $z\in H^1_0(\Omega)$ we select $\PPhi:=\Grad\grad w$ with 
$w\in H^2(\Omega)$ solving, for given $\varphi\in H^2(\Omega)|_\Gamma$,
\[
   \vdual{\Grad\grad w}{\Grad\grad \deltaz} + \vdual{w}{\deltaz} = -\dual{\varphi}{\deltaz}_\Gamma
   \quad\forall\deltaz\in H^2(\Omega).
\]
Then, $\div\Div\PPhi=-w$ and \eqref{z_trace} shows that
\[
   \vdual{\div\Div\PPhi}{z} - \vdual{\PPhi}{\Grad\grad z}
   =
   \dual{\varphi}{z}_\Gamma
   = 0.
\]
Since $\varphi\in H^2(\Omega)|_\Gamma$ was arbitrary, this yields $z|_\Gamma=0$.
Analogously, $\grad z|_\Gamma=0$ follows by selecting $\PPhi:=\Grad\grad w$ with
\[
   \vdual{\Grad\grad w}{\Grad\grad \deltaz} + \vdual{w}{\deltaz} = -\dual{\bvarphi}{\grad\deltaz}_\Gamma
   \quad\forall\deltaz\in H^2(\Omega)
\]
for arbitrary $\bvarphi\in \bH^1(\Omega)|_\Gamma$.

Next we show (ii).
The bound $\sum_{T\in\cT} \|\bv\|_\trggrad{\partial T}^2 \le \|\bv\|_\trggrad{0,\cS}^2$
holds by definition of the norms. To show the other inequality let
$\bv=(\bv_T)_T\in \bH^{3/2,1/2}_{00}(\cS)$ be given. There is $z\in H^2_0(\Omega)$ such that $\traceGG{}(z)=\bv$.
Furthermore, for $T\in\cT$, let $\wilde z_T\in H^2(T)$ be such that $\traceGG{T}(\wilde z_T)=\bv_T$ and
\(
   \|\bv_T\|_\trggrad{\partial T}
   =
   \|\wilde z_T\|_{2,T}.
\)
Then, $\wilde z$ defined by $\wilde z|_T:=\wilde z_T$ ($T\in\cT$) satisfies $\wilde z\in H^2(\cT)$,
and for any $\TTheta\in H(\div\Div\!,\Omega)$ it holds (cf.~\eqref{traces_id})
\begin{align*}
   \dual{\traceDD{}(\TTheta)}{\wilde z}_\cS
   &=
   \sum_{T\in\cT} \dual{\traceDD{T}(\TTheta)}{\wilde z_T}_{\partial T}
   =
   \sum_{T\in\cT} \dual{\traceGG{T}(\wilde z_T)}{\TTheta}_{\partial T}
   \\
   &=
   \sum_{T\in\cT}\dual{\traceGG{T}(z)}{\TTheta}_{\partial T}
   =
   \dual{\traceDD{}(\TTheta)}{z}_\cS
   =0
\end{align*}
by part (i). Again with (i) we conclude that $\wilde z\in H^2_0(\Omega)$.
Therefore,
\begin{align*}
   \sum_{T\in\cT} \|\bv|_{\partial T}\|_\trggrad{\partial T}^2
   &=
   \sum_{T\in\cT} \|\wilde z_T\|_{2,T}^2 = \|\wilde z\|_2^2
   \ge
   \|\bv\|_\trggrad{0,\cS}^2.
\end{align*}
This finishes the proof.
\end{proof}

By Proposition~\ref{prop_gg_jump}(i), for given $\bq\in \bH^{-3/2,-1/2}(\cS)$,
$\dual{\bq}{z}_\cS$ defines a functional that only depends on the jumps of $z$ and $\pwgrad z$
across the element interfaces and their traces on $\Gamma$. It will be denoted as
\begin{align} \label{jumpz}
   \jump{\cdot,\pwgrad\,\cdot}:\;
   \left\{\begin{array}{cll}
      H^2(\cT) & \to & \Bigl(\bH^{-3/2,-1/2}(\cS)\Bigr)'\\
      z        & \mapsto & \jump{z,\pwgrad z}(\bq) := \dual{\bq}{z}_\cS
   \end{array}\right.
\end{align}
with duality pairing defined in \eqref{tr_dd_dual}.
As before, this functional defines a semi-norm in $H^2(\cT)$,
\begin{align} \label{seminorm_z}
   \|\jump{z,\pwgrad z}\|_{(-3/2,-1/2,\cS)'}
   &:=
   \sup_{0\not=\bq\in \bH^{-3/2,-1/2}(\cS)}
   \frac{\dual{\bq}{z}_\cS}{\|\bq\|_{-3/2,-1/2,\cS}},
   \quad z\in H^2(\cT).
\end{align}

\begin{prop} \label{prop_gg_trace}
It holds the identity
\[
   \|\bv\|_{3/2,1/2,00,\cS} = \|\bv\|_\trggrad{0,\cS}
   \quad\forall\bv\in\bH^{3/2,1/2}_{00}(\cS).
\]
In particular,
\[
   \traceGG{}:\; H^2_0(\Omega)\to \bH^{3/2,1/2}_{00}(\cS)
\]
has unit norm and $\bH^{3/2,1/2}_{00}(\cS)$ is closed.
\end{prop}

\begin{proof}
The norm identity is obtained as in the proof of Proposition~\ref{prop_dd_trace}.
For $\bv=(\bv_T)_T\in\bH^{3/2,1/2}_{00}(\cS)$ we use Proposition~\eqref{prop_gg_jump}(ii)
and Lemma~\ref{la_tr_gg_norms} to calculate
\begin{align*}
   \|\bv\|_{3/2,1/2,00,\cS}^2
   &=
   \Bigl(\sup_{0\not=\TTheta\in H(\div\Div\!,\cT)}
      \frac{\sum_{T\in\cT}\dual{\bv_T}{\TTheta}_{\partial T}}{\|\TTheta\|_{\div\Div\!,\cT}}\Bigr)^2
   =
   \sum_{T\in\cT}
   \sup_{0\not=\TTheta\in H(\div\Div\!,T)} \frac{\dual{\bv_T}{\TTheta}_{\partial T}^2}{\|\TTheta\|_{\div\Div\!,T}^2}
   \\
   &=
   \sum_{T\in\cT} \|\bv_T\|_{3/2,1/2,00,\partial T}^2
   =
   \sum_{T\in\cT} \|\bv_T\|_\trggrad{\partial T}^2
   =
   \|\bv\|_\trggrad{0,\cS}^2.
\end{align*}
The space $\bH^{3/2,1/2}_{00}(\cS)$ is closed as the image of a bounded below operator.
\end{proof}

%===================================================================================================
\subsection{A Poincar\'e inequality in $H^2(\cT)$} \label{sec_gg_Poincare}

The definition of $\cC$ implies that it induces a self-adjoint isomorphism
$\LL_2^s(\Omega)\to\LL_2^s(\Omega)$. This fact will be used in the following.

Let us define a projection operator $\bP:\;\LL_2^s(\Omega)\to\cC\Grad\grad H^2_0(\Omega)$ by
\begin{align} \label{proj_P}
   \vdual{\bP(\TTheta)}{\cC\Grad\grad\deltaz}
   = \vdual{\TTheta}{\cC\Grad\grad\deltaz}\quad\forall\deltaz\in H^2_0(\Omega).
\end{align}
There is a mapping $\LL_2^s(\Omega)\ni\TTheta\mapsto \xi=\xi(\TTheta)\in H^2_0(\Omega)$ with
\begin{align} \label{proj_xi}
   \vdual{\Grad\grad\xi}{\cC\Grad\grad\deltaz} = \vdual{\TTheta}{\cC\Grad\grad\deltaz}
   \quad\forall\deltaz\in H^2_0(\Omega).
\end{align}
In other words,
\begin{align} \label{Pxi}
   \bP(\TTheta)=\Grad\grad \xi(\TTheta),\quad
   \div\Div\cC(\Grad\grad\xi-\TTheta)=0,\quad\text{and}\quad
   \cC\bigl(\Grad\grad\xi-\TTheta\bigr) \in H(\div\Div\!,\Omega).
\end{align}
In the next proposition we present a Poincar\'e inequality to bound
the $\|\cdot\|_{2,\cT}$-norm of a function from $H^2_0(\cT)$ by its jumps and
the projected piecewise iterated gradients. In the case of the Laplacian, such an estimate
is provided by \cite[Lemma 4.2]{DemkowiczG_11_ADM} and \cite[Lemma 3.3]{DemkowiczG_13_PDM}.
Our bound for $\|z\|$ can be proved almost identically to \cite[Lemma 4.2]{DemkowiczG_11_ADM},
switching from the Laplacian to the fourth-order operator and by using the trace operator $\traceDD{}$.

The second bound, for $\|\pwGrad\pwgrad z\|$, corresponds to \cite[Lemma 3.3]{DemkowiczG_13_PDM}
for the Laplacian. The proof of \cite[Lemma 3.3]{DemkowiczG_13_PDM} uses
a projection operator like $\bP$, together with the technical lemma
\cite[Lemma 4.3]{DemkowiczG_11_ADM}. This lemma provides an estimate by norms of
jumps of natural and essential traces (traces that correspond to natural and essential boundary conditions
on elements) and, moreover, uses a Helmholtz decomposition for
its proof. Whereas there is a Helmholtz decomposition for $H(\div\Div\!,\Omega)$,
cf.~\cite[Section~4.1]{AmaraCPC_02_BMM} and \cite[Theorem 4.2]{RafetsederZ_DRK},
the use of jumps of natural traces in $H(\div\Div\!,\cT)$ appears to be too complicated in our case.
We present a shorter technique that is based on the projection operator $\bP$ without using
jumps of natural traces.

\begin{prop} \label{prop_gg_Poincare}
The following Poincar\'e inequality holds,
\[
   \|z\|_{2,\cT} \lesssim \|\bP(\pwGrad\pwgrad z)\| + \|\jump{z,\pwgrad z}\|_{(-3/2,-1/2,\cS)'}
   \quad\forall z\in H^2(\cT).
\]
Here, the implicit constant is independent of the underlying mesh $\cT$,
it only depends on $\Omega$ and $\cC$.
\end{prop}

\begin{proof}
We start by proving
\[
   \|z\| \lesssim \|\bP(\pwGrad\pwgrad z)\| + \|\jump{z,\pwgrad z}\|_{(-3/2,-1/2,\cS)'}
   \quad\forall z\in H^2(\cT).
\]
For given $z\in H^2(\cT)$ let $\phi\in H^2_0(\Omega)$ be the solution to
\[
   \div\Div\cC\Grad\grad\phi = z\quad\text{in}\quad\Omega.
\]
Then $\|\cC\Grad\grad\phi\|\le C \|z\|$ for a constant $C>0$ that only depends on $\Omega$
and $\cC$, and we obtain, by using definition \eqref{seminorm_z} and Proposition~\ref{prop_dd_trace},
\begin{align*}
   \|z\|^2
   &= \vdual{\div\Div\cC\Grad\grad\phi}{z}
    = \vdual{\cC\Grad\grad\phi}{\pwGrad\pwgrad z} + \dual{\traceDD{}(\cC\Grad\grad\phi)}{z}_\cS
   \\
   &\le
   C \|z\| \|\bP(\pwGrad\pwgrad z)\|
   + \|\jump{z,\pwgrad z}\|_{(-3/2,-1/2,\cS)'}
     \Bigl( \|\cC\Grad\grad\phi\|^2 + \|\div\Div\cC\Grad\grad\phi\|^2 \Bigr)^{1/2}
   \\
   &\lesssim
   \|z\| \Bigl( \|\bP(\pwGrad\pwgrad z)\| + \|\jump{z,\pwgrad z}\|_{(-3/2,-1/2,\cS)'} \Bigr).
\end{align*}
This proves the bound for $\|z\|$.
It remains to show that
\[
   \|\pwGrad\pwgrad z\| \lesssim \|\bP(\pwGrad\pwgrad z)\| + \|\jump{z,\pwgrad z}\|_{(-3/2,-1/2,\cS)'}
   \quad\forall z\in H^2(\cT).
\]
For given $z\in H^2(\cT)$ let $\xi=\xi(\pwGrad\pwgrad z)$, cf.~\eqref{proj_xi}.
By definition \eqref{proj_P} and relation \eqref{Pxi} it holds
\begin{align} \label{pf_gg_Poincare_1}
   \vdual{\pwGrad\pwgrad z}{\cC\Grad\grad\xi}
   = \vdual{\Grad\grad\xi}{\cC\Grad\grad\xi}
   \lesssim \|\Grad\grad\xi\|^2 = \|\bP(\pwGrad\pwgrad z)\|^2.
\end{align}
By \eqref{Pxi} we have
$\div\Div\cC\pwGrad\pwgrad(z-\xi)=0$ and $\cC\pwGrad\pwgrad(z-\xi)\in H(\div\Div\!,\Omega)$.
Recalling \eqref{trT_dd}, \eqref{tr_dd} with $\TTheta=\cC\pwGrad\pwgrad(z-\xi)$ we conclude that
\begin{align} \label{pf_gg_Poincare_2}
   \vdual{\pwGrad\pwgrad z}{\cC\pwGrad\pwgrad(z-\xi)}
   &=
   \vdual{z}{\div\Div\cC\pwGrad\pwgrad(z-\xi)} - \dual{\traceDD{}(\cC\pwGrad\pwgrad(z-\xi))}{z}_\cS
   \nonumber\\
   &=
   - \dual{\traceDD{}(\cC\pwGrad\pwgrad(z-\xi))}{z}_\cS.
\end{align}
Combination of \eqref{pf_gg_Poincare_1} and \eqref{pf_gg_Poincare_2} yields
\begin{align*}
   \|\pwGrad\pwgrad z\|^2
   &\lesssim \vdual{\pwGrad\pwgrad z}{\cC\pwGrad\pwgrad z}
   =
   \vdual{\pwGrad\pwgrad z}{\cC\Grad\grad\xi} + \vdual{\pwGrad\pwgrad z}{\cC\pwGrad\pwgrad(z-\xi)}
   \\
   &\lesssim
   \|\bP(\pwGrad\pwgrad z)\|\, \|\Grad\grad\xi\|
   -
   \dual{\traceDD{}(\cC\pwGrad\pwgrad(z-\xi))}{z}_\cS.
\end{align*}
We finish the proof by bounding $\|\Grad\grad\xi\|\lesssim \|\pwGrad\pwgrad z\|$
by stability of problem \eqref{proj_xi},
%(or by boundedness of $\bP$ with unit norm),
and by applying as before, Proposition~\ref{prop_dd_trace} in combination with definition~\eqref{seminorm_z}.
This gives
\begin{align*}
\lefteqn{
   \dual{\traceDD{}(\cC\pwGrad\pwgrad(z-\xi))}{z}_\cS
   \le
   \|\cC\pwGrad\pwgrad(z-\xi)\|_{\div\Div}\|\jump{z,\pwgrad z}\|_{(-3/2,-1/2,\cS)'}
}
   \\
   &=
   \|\cC\pwGrad\pwgrad(z-\xi)\| \|\jump{z,\pwgrad z}\|_{(-3/2,-1/2,\cS)'}
   \lesssim
   \|\pwGrad\pwgrad z\| \|\jump{z,\pwgrad z}\|_{(-3/2,-1/2,\cS)'}.
\end{align*}
\end{proof}

%===================================================================================================
\section{Variational formulation and DPG method} \label{sec_uw}

Let us return to our preliminary formulation \eqref{VFa}. We now know that we have to interpret
the interface terms as
\[
   \sum_{T\in\cT} \dual{\nn\cdot\Div\MM}{z}_{\partial T}
   - \sum_{T\in\cT} \dual{\MM\nn}{\grad z}_{\partial T}
   =
   \dual{\traceDD{}(\MM)}{z}_\cS
\]
and
\[
     \sum_{T\in\cT} \dual{\nn\cdot\Div\TTheta}{u}_{\partial T}
   - \sum_{T\in\cT} \dual{\TTheta\nn}{\grad u}_{\partial T}
   =
   \dual{\traceGG{}(u)}{\TTheta}_\cS.
\]
Introducing the independent trace variables
$\tQ:=\traceDD{}(\MM)$, $\tu:=\traceGG{}(u)$, and spaces
\begin{align*}
   &\UU := L_2(\Omega)\times\LL_2^s(\Omega)\times \bH^{3/2,1/2}_{00}(\cS) \times \bH^{-3/2,-1/2}(\cS),\\
   &\VV := H^2(\cT)\times H(\div\Div\!,\cT)
\end{align*}
with respective norms
\begin{align*}
   \|(u,\MM,\tu,\tQ)\|_\UU^2
   &:=
   \|u\|^2 + \|\MM\|^2 + \|\tu\|_{3/2,1/2,00,\cS}^2 + \|\tQ\|_{-3/2,-1/2,\cS}^2,
   \\
   \|(z,\TTheta)\|_\VV^2
   &:=
   \|z\|_{2,\cT}^2 + \|\TTheta\|_{\div\Div\!,\cT}^2,
\end{align*}
our ultraweak variational formulation of \eqref{prob} is:
\emph{Find $(u,\MM,\tu,\tQ)\in \UU$ such that}
\begin{align} \label{VF}
   b(u,\MM,\tu,\tQ;z,\TTheta) = L(z,\TTheta)
   \quad\forall (z,\TTheta)\in\VV.
\end{align}
Here,
\begin{align} \label{b}
   b(u,\MM,\tu,\tQ;z,\TTheta)
   :=
   &\vdual{\MM}{\pwGrad\pwgrad z+\cCinv\TTheta}
   + \vdual{u}{\pwdiv\pwDiv\TTheta}
   + \dual{\tQ}{z}_\cS - \dual{\tu}{\TTheta}_\cS,
\end{align}
and
\begin{align*}
   L(z,\TTheta) := -\vdual{f}{z}.
\end{align*}
Note that the skeleton dualities in \eqref{b} are defined by
\eqref{tr_dd_dual} and \eqref{tr_gg_dual}.
One of our main results is the following theorem.

\begin{theorem} \label{thm_stab}
For any function $f\in L_2(\Omega)$ (or any functional $L\in\VV'$) there exists
a unique and stable solution $(u,\MM,\tu,\tQ)\in \UU$ to \eqref{VF},
\[
   \|u\| + \|\MM\| + \|\tu\|_{3/2,1/2,00,\cS} + \|\tQ\|_{-3/2,-1/2,\cS}
   \lesssim
   \|f\|\quad\text{(or $\|L\|_{\VV'}$)}
\]
with a hidden constant that is independent of $f$ (or $L$) and $\cT$.
\end{theorem}

A proof of this theorem is given in Section~\ref{sec_pf}.

Now, the DPG method with optimal test functions consists in solving \eqref{VF}
within discrete spaces $\UU_h\subset\UU$ and $\ttt(\UU_h)\subset\VV$.
Here,  $\ttt:\;\UU\to\VV$ is the \emph{trial-to-test operator}, defined by
\[
   \ip{\ttt(\uu)}{\vv}_\VV = b(\uu,\vv)\quad\forall\vv\in\VV
\]
with inner product $\ip{\cdot}{\cdot}$ in $\VV$.

Then, for given finite-dimensional space $\UU_h\subset\UU$, the discrete method is:
\emph{Find $\uu_h\in\UU_h$ such that}
\begin{align} \label{DPG}
   b(\uu_h,\ttt\bdeltau) = L(\ttt\bdeltau) \quad\forall\bdeltau\in\UU_h.
\end{align}
It is a minimum residual method that delivers the best approximation
in the \emph{energy norm} (or residual norm)
$\|\cdot\|_\EE:=\|B(\cdot)\|_{\VV'}$, cf., e.g.,~\cite{DemkowiczG_11_ADM}.
Here, $B:\;\UU\to\VV'$ is the operator induced by the bilinear form $b(\cdot,\cdot)$.

Our second main result is the quasi-optimal convergence of the DPG scheme \eqref{DPG}.

\begin{theorem} \label{thm_DPG}
Let $f\in L_2(\Omega)$ be given. For any finite-dimensional subspace $\UU_h\subset\UU$
there exists a unique solution $\uu_h\in\UU_h$ to \eqref{DPG}. It satisfies the quasi-optimal
error estimate
\[
   \|\uu-\uu_h\|_\UU \lesssim \|\uu-\bw\|_\UU
   \quad\forall\bw\in\UU_h
\]
with a hidden constant that is independent of $f$, $\cT$ and $\UU_h$.
\end{theorem}

A proof of this theorem is given in Section~\ref{sec_pf}.

%===================================================================================================
\section{Adjoint problem and proofs of Theorems~\ref{thm_stab},~\ref{thm_DPG}} \label{sec_adj}

As discussed in the introduction, key step to show well-posedness of the variational formulation
\eqref{VF} is to show stability of its adjoint problem, which we formulate next.

\emph{Find $z\in H^2(\cT)$ and $\TTheta\in H(\div\Div\!,\cT)$ such that}
\begin{subequations} \label{adj}
\begin{alignat}{2}
    \pwdiv\pwDiv\TTheta         &= g  && \ \in L_2(\Omega)\label{a1},\\
    \cCinv\TTheta + \pwGrad\pwgrad z  &= \bH  &&\ \in \LL_2^s(\Omega)\label{a2},\\
    \jump{\TTheta\nn,\nn\cdot\pwDiv\TTheta} 
                                &= \br  && \ \in \Bigl(\bH^{3/2,1/2}_{00}(\cS)\Bigr)'\label{jtheta},\\
    \jump{z,\pwgrad z}          &= \bj  && \ \in \Bigl(\bH^{-3/2,-1/2}(\cS)\Bigr)'\label{jz}.
\end{alignat}
\end{subequations}
Here, initially, the data $g$, $\bH$, $\br$, and $\bj$ are obtained
as indicated from the given (arbitrary) function $(z,\TTheta)\in\VV$.
Recall \eqref{jumpQQ} and \eqref{jumpz} for the definition of the jumps.

Proving well-posedness of \eqref{adj} means that we separate the data
from the particular test functions $z$, $\TTheta$.
Then, the functionals on the right-hand sides of \eqref{adj} are arbitrary elements
of the corresponding spaces as indicated.
Specifically, by definition of the dual spaces in \eqref{jtheta}, \eqref{jz},
the functionals $\br$ and $\bj$ stem from corresponding functions (now using different symbols)
$\TTheta_\br\in H(\div\Div\!,\cT)$ and $z_\bj\in H^2(\cT)$, respectively, so that
the following definitions apply.
\begin{align*}
   \text{Given}\ \bv\in \bH^{3/2,1/2}_{00}(\cS),
   &\quad\br(\bv) := \dual{\bv}{\TTheta_\br}_\cS
   &&\text{(according to \eqref{jumpQQ})},
   \\
   \text{and given}\ \bq \in \bH^{-3/2,-1/2}(\cS),
   &\quad\bj(\bq) := \dual{\bq}{z_\bj}_\cS
   &&\text{(according to \eqref{jumpz})}.
\end{align*}
Of course, the functions $\TTheta_\br$, $z_\bj$ are not unique but the induced functionals are.
As indicated in \eqref{jtheta}, \eqref{jz}, the functionals $\br$ and $\bj$ are measured in
dual norms $\|\cdot\|_{(3/2,1/2,00,\cS)'}$ and\linebreak
$\|\cdot\|_{(-3/2,-1/2,\cS)'}$, respectively, see \eqref{seminorm_QQ}, \eqref{seminorm_z}.

%===================================================================================================
\subsection{Well-posedness of the adjoint problem} \label{sec_adj_well}

In the following we again use that $\cC$ induces a self-adjoint isomorphism
$\LL_2^s(\Omega)\to\LL_2^s(\Omega)$.

Combining \eqref{a1} and \eqref{a2} we obtain, in distributional form,
\begin{align} \label{zHg_distr}
   \pwdiv\pwDiv\bigl(\cC\pwGrad\pwgrad z\bigr) = \pwdiv\pwDiv\bigl(\cC\bH\bigr) - g.
\end{align}
Testing with $\deltaz\in H^2_0(\Omega)$ and twice integrating piecewise by parts gives
\begin{align*}
   \vdual{\cC\pwGrad\pwgrad z}{\Grad\grad\deltaz}
   + \dual{\traceGG{}(\deltaz)}{\cC\pwGrad\pwgrad z}_\cS
   =
     \vdual{\cC\bH}{\Grad\grad\deltaz}
   + \dual{\traceGG{}(\deltaz)}{\cC\bH}_\cS
   - \vdual{g}{\deltaz}
\end{align*}
(recall the trace operator $\traceGG{}$ from \eqref{tr_gg}).
Now, by \eqref{a2}, $\cC\bigl(\bH-\pwGrad\pwgrad z\bigr)=\TTheta\in H(\div\Div\!,\cT)$ so that the combined
interface terms are well defined via \eqref{tr_gg},
and coincide with the jumps associated to $\TTheta$,
\[
   \dual{\traceGG{}(\deltaz)}{\cC\bigl(\bH-\pwGrad\pwgrad z\bigr)}_\cS
   =
   \dual{\traceGG{}(\deltaz)}{\TTheta}_\cS
   =
   \jump{\TTheta\nn,\nn\cdot\pwDiv\TTheta}(\traceGG{}(\deltaz)),
\]
cf.~\eqref{jumpQQ}. Taking into account \eqref{jtheta} and \eqref{jz}, the $z$-component
of the solution to \eqref{adj} satisfies the following reduced adjoint problem.

\emph{Given $g\in L_2(\Omega)$, $\bH\in\LL_2^s(\Omega)$, $\br\in\bigl(\bH^{3/2,1/2}_{00}(\cS)\bigr)'$,
and $\bj\in \bigl(\bH^{-3/2,-1/2}(\cS)\bigr)'$, find $z\in H^2(\cT)$ such that}
\begin{subequations} \label{saddle}
\begin{alignat}{2}
   \vdual{\cC\pwGrad\pwgrad z}{\Grad\grad\deltaz}
   &=
     \vdual{\cC\bH}{\Grad\grad\deltaz}
   - \vdual{g}{\deltaz}
   + \br(\traceGG{}(\deltaz))
   &&\quad\forall \deltaz\in H^2_0(\Omega),
   \label{s1}\\
   \jump{z,\pwgrad z}(\bdeltaq) &= \bj(\bdeltaq)
   &&\quad\forall\bdeltaq\in \bH^{-3/2,-1/2}(\cS).
   \label{s2}
\end{alignat}
\end{subequations}

\begin{lemma} \label{la_saddle}
Problem \eqref{saddle} has a unique solution $z\in H^2(\cT)$. It satisfies
\[
   \|z\|_{2,\cT}
   \lesssim
   \|g\| + \|\bH\| + \|\br\|_{(3/2,1/2,00,\cS)'} + \|\bj\|_{(-3/2,-1/2,\cS)'}.
\]
\end{lemma}

\begin{proof}
Adding relations \eqref{s1}, \eqref{s2} we represent \eqref{saddle} with the notation
\[
   a(z;\deltaz,\bdeltaq) = l(\deltaz,\bdeltaq).
\]
We show that $a(\cdot;\cdot)$ and $l(\cdot)$ are bounded and that $a(\cdot;\cdot)$ satisfies the
required inf-sup conditions.

The boundedness of $l(\cdot)$ is immediate by duality of involved norms,
\begin{align*}
   l(\deltaz,\bdeltaq)
   &=
     \vdual{\cC\bH}{\Grad\grad\deltaz}
   - \vdual{g}{\deltaz}
   + \br(\traceGG{}(\deltaz))
   + \bj(\bdeltaq)
   \\
   &\le
     \|\cC\bH\| \|\Grad\grad\deltaz\|
   + \|g\| \|\deltaz\|
   + \|\br\|_{(3/2,1/2,00,\cS)'} \bigl(\|\deltaz\|^2 + \|\Grad\grad\deltaz\|^2\bigr)^{1/2}
   \\
   &\quad
   + \|\bj\|_{(-3/2,-1/2,\cS)'} \|\bdeltaq\|_{-3/2,-1/2,\cS}
   \\
   \lesssim
   \Bigl(
      \|\bH\| &+ \|g\| + \|\br\|_{(3/2,1/2,00,\cS)'} + \|\bj\|_{(-3/2,-1/2,\cS)'}
   \Bigr)
   \Bigl(\|\deltaz\|_2 + \|\bdeltaq\|_{-3/2,-1/2,\cS}\Bigr).
\end{align*}
The boundedness of $a(\cdot;\cdot)$ is also immediate by using definition \eqref{jumpz}
and duality norm \eqref{norm_tr_dd},
\begin{align*}
   \jump{z,\pwgrad z}(\bdeltaq) \le \|\bdeltaq\|_{-3/2,-1/2,\cS} \|z\|_{2,\cT}.
\end{align*}
It remains to show the inf-sup conditions.

Let $\deltaz\in H^2_0(\Omega)$ and $\bdeltaq\in H^{-3/2,-1/2}(\cS)$ with
$a(z;\deltaz,\bdeltaq)=0$ for any $z\in H^2(\cT)$.
Selecting $z:=\deltaz\in H^2_0(\Omega)$, Proposition~\ref{prop_gg_jump}(i) shows that
$\jump{z,\pwgrad z}(\bdeltaq)=\dual{\bdeltaq}{z}_\cS=0$
(recall \eqref{tr_dd_dual} for the definition of the duality) so that
\[
   a(z;\deltaz,\bdeltaq)
   =
   \vdual{\cC\Grad\grad\deltaz}{\Grad\grad\deltaz} + \jump{z,\pwgrad z}(\bdeltaq)
   =
   \vdual{\cC\Grad\grad\deltaz}{\Grad\grad\deltaz}
   \gtrsim
   \|\Grad\grad\deltaz\|^2=0,
\]
that is, $\deltaz=0$. Using the observed relation for the jump of $z$, it follows that
\[
    a(z;\deltaz,\bdeltaq) = \dual{\bdeltaq}{z}_\cS = 0\quad\forall z\in H^2(\cT),
\]
i.e.,
\[
   \|\bdeltaq\|_{-3/2,-1/2,\cS} = \|\bdeltaq\|_\trddiv{\cS} = 0
\]
by \eqref{norm_tr_dd} and Proposition~\ref{prop_dd_trace}. Therefore, $\bdeltaq=0$.

Finally we check the inf-sup condition,
\begin{align*}
   \sup_{0\not=(\deltaz,\bdeltaq)\in H^2_0(\Omega)\times H^{-3/2,-1/2}(\cS)}
   &\frac {\bigl(\vdual{\cC\pwGrad\pwgrad z}{\Grad\grad\deltaz} + \jump{z,\pwgrad z}(\bdeltaq)\bigr)^2}
         {\|\cC\Grad\grad\deltaz\|^2 + \|\bdeltaq\|_{-3/2,-1/2,\cS}^2}
   \\
   &=
   \|\bP(\pwGrad\pwgrad z)\|^2 + \|\jump{z,\pwgrad z}\|_{(-3/2,-1/2,\cS)'}^2
   \quad\forall z\in H^2(\cT),
\end{align*}
cf.~\eqref{proj_P}. The result follows by the equivalence of the norms
$\|\deltaz\|_2$ and $\|\cC\Grad\grad\deltaz\|$ for $\deltaz\in H^2_0(\Omega)$,
and an application of the Poincar\'e inequality (Proposition~\ref{prop_gg_Poincare}).
\end{proof}

Having analyzed the reduced adjoint problem \eqref{saddle}, we are ready to prove the well-posedness
of the full adjoint problem \eqref{adj}.

\begin{prop} \label{prop_adj}
For arbitrary $g\in L_2(\Omega)$, $\bH\in\LL_2^s(\Omega)$,
$\br\in\bigl(\bH^{3/2,1/2}_{00}(\cS)\bigr)'$, and $\bj\in \bigl(\bH^{-3/2,-1/2}(\cS)\bigr)'$,
the adjoint problem \eqref{adj} has a unique solution $(z,\TTheta)\in\VV$. It satisfies
\[
   \|z\|_{2,\cT} + \|\TTheta\|_{\div\Div\!,\cT}
   \lesssim
   \|g\| + \|\bH\| + \|\br\|_{(3/2,1/2,00,\cS)'} + \|\bj\|_{(-3/2,-1/2,\cS)'}.
\]
\end{prop}

\begin{proof}
By construction, the $z$-component of any solution
$(z,\TTheta)\in\VV$ of \eqref{adj} satisfies \eqref{saddle}, which is uniquely solvable
by Lemma~\ref{la_saddle}. Therefore, the $z$-component of \eqref{adj} is unique.
Starting with the solution $z\in H^2(\cT)$ to \eqref{saddle}, we show that this
leads to a unique solution $(z,\TTheta)\in\VV$ of \eqref{adj}, satisfying the stated bound.
By relation \eqref{s2}, $z$ satisfies \eqref{jz}.
According to Lemma~\ref{la_saddle}, $z$ also satisfies the required bound.

It remains to construct $\TTheta$ and to bound its norm.
We define $\TTheta:=\cC\bigl(\bH-\pwGrad\pwgrad z\bigr)\in\LL_2^s(\Omega)$, thus satisfying (uniquely)
\eqref{a2}.
Using the bound for $\|\pwGrad\pwgrad z\|$, we also see that
$\|\TTheta\|\lesssim \|g\| + \|\bH\| + \|\br\|_{(3/2,1/2,00,\cS)'} + \|\bj\|_{(-3/2,-1/2,\cS)'}$.

Now, \eqref{zHg_distr} shows that
\[
   \pwdiv\pwDiv\TTheta = \pwdiv\pwDiv\cC\bigl(\bH-\pwGrad\pwgrad z\bigr) = g
\]
holds first in distributional sense, and then in $L_2(\Omega)$ by the regularity of $g$. This is \eqref{a1}
and also concludes the proof of the bound for $\|\TTheta\|_{\div\Div\!,\cT}$.

It remains to show \eqref{jtheta}. Let $\bv\in\bH^{3/2,1/2}_{00}(\cS)$ with $\bv=\traceGG{}(v)$
for $v\in H^2_0(\Omega)$.
Recalling the definitions \eqref{jumpQQ}, \eqref{tr_gg_dual}, \eqref{tr_gg}, and \eqref{trT_gg},
we calculate with the previous relations for $\TTheta$ and \eqref{s1},
\begin{align*}
   \jump{\TTheta\nn,\nn\cdot\pwDiv\TTheta}(\bv)
   &=
   \dual{\bv}{\TTheta}_\cS
   =
   \dual{\traceGG{}(v)}{\TTheta}_\cS
   =
   \vdual{\pwdiv\pwDiv\TTheta}{v} - \vdual{\TTheta}{\Grad\grad v}
   \\
   &=
   \vdual{g}{v} - \vdual{\cC(\bH-\pwGrad\pwgrad z)}{\Grad\grad v}
   =
   \br(\traceGG{}(v))
   =
   \br(\bv).
\end{align*}
This shows \eqref{jtheta}, and finishes the proof.
\end{proof}
%===================================================================================================
\subsection{Proofs of Theorems~\ref{thm_stab},~\ref{thm_DPG}} \label{sec_pf}

We are ready to prove our main results. To show Theorem~\ref{thm_stab},
it is enough to check the standard properties.
\begin{enumerate}
\item {\bf Boundedness of the functional.} This is immediate since, for $f\in L_2(\Omega)$, it holds
      $L(z)\le \|f\|\,\|z\|\le \|f\|\,\bigl(\|z\|_{2,\cT} + \|\TTheta\|_{\div\Div\!,\cT}\bigr)$
      for any $(z,\TTheta)\in\VV$.
\item {\bf Boundedness of the bilinear form.} 
      The bound $b(\uu,\vv)\lesssim \|\uu\|_\UU \|\vv\|_\VV$ for all $\uu\in\UU$ and
      $\vv\in\VV$ is also immediate by definition of the
      norms in $\UU$ and $\VV$, cf. the corresponding functional spaces in
      \eqref{a1}--\eqref{jz}.
\item {\bf Injectivity.} If $\uu\in\UU$ with $b(\uu,\vv)=0\ \forall \vv\in\VV$ then
      $\uu=0$, as can be seen as follows. For given $\uu=(u,\MM,\tu,\tQ)\in\UU$
      we select $g=u$, $\bH=\MM$, and let $\bj\in (\bH^{-3/2,-1/2}(\cS))'$
      and $\br\in (\bH^{3/2,1/2}_{00}(\cS))'$ be the Riesz representatives of
      $\tQ\in\bH^{-3/2,-1/2}(\cS)$ and $-\tu\in \bH^{3/2,1/2}_{00}(\cS)$, respectively.
      According to Proposition~\ref{prop_adj}, there exists $\vv\in\VV$ that satisfies
      the adjoint problem \eqref{adj} with these functionals. It also yields
      \[
         b(\uu,\vv) = \|u\|^2 + \|\MM\|^2 + \|\tu\|_{3/2,1/2,00,\cS}^2 + \|\tQ\|_{-3/2,-1/2,\cS}^2
         =0
      \]
      which proves that $\uu=0$.
\item {\bf Inf-sup condition.} For given $\vv=(z,\TTheta)\in\VV$ let $g$, $\bH$,
      $\bj$, and $\br$ be defined by \eqref{adj}. Then, by Proposition~\ref{prop_adj},
\begin{align*}
   \sup_{0\not=\uu\in\UU} \frac {b(\uu,\vv)}{\|\uu\|_\UU}
   &=
   \sup_{0\not=\uu\in\UU}
   \frac {\vdual{u}{g} + \vdual{\MM}{\bH} - \br(\tu) + \bj(\tQ)}
   {\bigl(\|u\|^2 + \|\MM\|^2 + \|\tu\|_{3/2,1/2,00,\cS}^2 + \|\tQ\|_{-3/2,-1/2,\cS}^2\bigr)^{1/2}}
   \\
   &=
   \Bigl(\|g\|^2 + \|\bH\|^2 + \|\br\|_{(3/2,1/2,00,\cS)'}^2 + \|\bj\|_{(-3/2,-1/2,\cS)'}^2
   \Bigr)^{1/2}
   \gtrsim
   \|\vv\|_{\VV}
\end{align*}
   with an implicit constant that is independent of $\vv$ and $\cT$.
\end{enumerate}
This proves Theorem~\ref{thm_stab}.

Recall that the DPG method delivers the best approximation in the energy norm $\|\cdot\|_\EE$,
\[
   \|\uu-\uu_h\|_\EE = \min\{\|\uu-\bw\|_\EE;\; \bw\in\UU_h\}.
\]
Therefore, to show Theorem~\ref{thm_DPG}, it is enough to prove the equivalence of the
energy norm and the norm $\|\cdot\|_\UU$.
The bound $\|\uu\|_\EE\lesssim\|\uu\|_\UU$ is equivalent to the boundedness
of $b(\cdot,\cdot)$, which we have just checked.
By definition of $\|\cdot\|_\EE=\|B(\cdot)\|_{\VV'}$, the other inequality,
$\|\uu\|_\UU\lesssim \|\uu\|_\EE$ for all $\uu\in\UU$, is equivalent to
the stability of the adjoint problem \eqref{adj}, which has been shown by
Proposition~\ref{prop_adj}. We have thus shown Theorem~\ref{thm_DPG}.

%===================================================================================================
% Numerics
%===================================================================================================
%\input{numeric}
%%%%%%%%%%%%%%%%%%%%%%%%%%%%%%%%%%%%%%%%%%%%%%%%%%%%%%%%
\section{Discretization and numerical examples} \label{sec_num}

In this section we discuss the construction of low-order discrete spaces,
some implementational aspects, and present numerical tests.
Throughout, we consider $\di=2$ and use regular triangular meshes $\cT$ of shape-regular elements,
\begin{align*}
  \sup_{T\in\cT} \frac{\diam(T)^2}{|T|} \leq C_\mathrm{shape}.
\end{align*}
As usual we denote by $h:=h_\cT := \max_{T\in\cT} \diam(T)$ the discretization parameter.

\subsection{Discrete spaces}\label{sec:discretespaces}
For $T\in\cT$, let $P^p(T)$ denote the space of polynomials on $T$ of order less than or
equal to $p\in\N_0$ and define
\begin{align*}
  P^p(\cT) := \{v\in L_2(\Omega);\; v|_T\in P^p(T) \,\forall T\in\cT\}.
\end{align*}
We set $\PP^p(T):=P^p(T)^{2\times 2}$ and $\PP^p(\cT) := P^p(\cT)^{2\times 2}$.
We seek approximations of the field variables $(u,\MM)\in L_2(\Omega)\times \LL_2^s(\Omega)$
in piecewise polynomial spaces,
\begin{align*}
  (u_h,\MM_h) \in P^0(\cT) \times (\PP^0(\cT)\cap \LL_2^s(\Omega)).
\end{align*}
In the following we use the notation for edges, nodes and their sets as introduced at the beginning
of \S\ref{sec_jump2_dd}. Specifically,
$\cE_T$ denotes the set of edges of $T$ and $\cE := \bigcup_{T\in\cT} \cE_T$.
Let $P^p(E)$ denote the space of polynomials on $E\in\cE$ and define
\begin{align*}
  P^p(\cE_T) := \{v\in L_2(\partial T);\; v|_E \in P^p(E)\ \forall E\in \cE_T\},
  \quad T\in\cT.
\end{align*}
The definition of conforming discrete spaces for the skeleton variables $(\tu,\tQ)$
is a little bit more involved. For a simpler representation we only consider lowest-order spaces.
We start by defining, for $T\in\cT$, the local space
\begin{align*}
  \widehat U_{\trggrad{\partial T}}
  :=
  \traceGG{T}\left(\{v\in H^2(T);\; \Delta^2 v + v = 0, \, v|_{\partial T} \in P^3(\cE_T), \, 
  \nn\cdot\grad v|_{\partial T}\in P^1(\cE_T)\}\right).
\end{align*}
Let $\cN_T$ denote the vertex set of $T\in\cT$ and set $\cN := \bigcup_{T\in\cT} \cN_T$.
We associate the following degrees of freedom to a triangle $T$ and the space $\widehat U_{\trggrad{\partial T}}$,
\begin{align}\label{eq:dof:tU}
  \{(v(e),\grad v(e));\; e\in\cN_T\}.
\end{align}
Observe that these degrees of freedom define a unique function in $\widehat U_{\trggrad{\partial T}}$.
The corresponding global discrete space is then defined by
\begin{align*}
  \widehat U_{\trggrad\cS}
    := \{\tv\in \bH^{3/2,1/2}(\cS);\; \tv|_{\partial T}  \in \widehat U_{\trggrad{\partial T}}\ \forall T\in\cT\}
\end{align*}
with associated global degrees $\{(v(e),\grad v(e));\; e\in\cN\}$.
To get a subspace of $\bH^{3/2,1/2}_{00}(\cS)$ we set the degrees of freedom corresponding
to boundary vertices to zero, leading to the space
\begin{align*}
%  \widehat U_{\trggrad\cS,00} := \traceGG{}(H_0^2(\Omega)) \cap \widehat U_{\trggrad\cS}
   \widehat U_\cS := \traceGG{}(H_0^2(\Omega)) \cap \widehat U_{\trggrad\cS}
\end{align*}
with dimension $3\#\cN_0$.

\begin{remark}\label{rem:approxTraceGG}
  Our definition of the skeleton spaces is closely related to the traces of spaces used in virtual element methods.
  In fact, for the present case the trace of the space defined in~\cite[\S 4.2]{BrezziM_13_VEM}
  is the same as $\widehat U_\cS$. 
  In particular, we get the approximation property, cf.~\cite[Remark~4.6]{BrezziM_13_VEM},
  \begin{align*}
    \min_{\tu_h\in\widehat U_\cS} \norm{\tu-\tu_h}{3/2,1/2,00,\cS} \leq C h |u|_{H^3(\Omega)}
  \end{align*}
  where $\tu = \traceGG{}(u)$ and
  $C>0$ is a generic $\cT$-independent constant.
  Let us note that $\widehat U_\cS$ coincides also with the trace of the (reduced)
  Hsieh--Clough--Tocher composite finite element space, cf.~\cite{Ciarlet_78_IEE}.
\end{remark}

It remains to construct a finite-dimensional subspace for the approximation of $\tQ \in\bH^{-3/2,-1/2}(\cS)$.
For $T\in\cT$ we define the local (volume) space
\begin{align}
\begin{split}\label{eq:defUddivT}
  U_{\trddiv{T}} := \{\TTheta\in &\,\cH(\div\Div\!,T);\; \Grad\grad\div\Div \TTheta +\TTheta = 0,\\
  &\bigl(\nn\cdot\Div\TTheta
         + \partial_{\bt,\cE_T}
                    (\bt\cdot\TTheta\nn)\bigr)|_{\partial T}\in P^0(\cE_T),\quad
  \nn\cdot\TTheta\nn|_{\partial T} \in P^0(\cE_T) \}.
\end{split}
\end{align}
Here, $\partial_{\bt,\cE_T}$ denotes the tangential derivative operator that is taken piecewise
on the edges of $\partial T$, cf.~Remark~\ref{rem:traceOp2D}.
To this space we associate the moments and point values
\begin{subequations}\label{eq:dofDDlocal}
\begin{align}
  &\alpha_E := \dual{\nn\cdot\Div\TTheta + \partial_\bt (\bt\cdot\TTheta\nn)}{1}_E
  \hspace{-7em}&&(E\in\cE_T),\label{eq:dofDDlocal_a}\\
  &\beta_E := \dual{\nn\cdot\TTheta\nn}{1}_E
  &&(E\in\cE_T),\label{eq:dofDDlocal_b}\\
  &\gamma_e := \jjump{\TTheta}_{\partial T}(e)
  &&(e\in \cN_T),\label{eq:dofDDlocal_c}
\end{align}
\end{subequations}
cf.~\eqref{traces_tn_2d}, \eqref{jjump_node}.

\begin{lemma}\label{lem:dofDDlocal}
  The degrees of freedom~\eqref{eq:dofDDlocal} define a unique element in $U_{\trddiv{T}}$ and vice versa.
\end{lemma}
\begin{proof}
  We prove that~\eqref{eq:dofDDlocal} define a unique functional $\ell(\cdot)$ 
  on $H^2(T)$ that vanishes for $z\in H_0^2(T)$. Then, the proof of Lemma~\ref{la_tr_dd_norms} shows that this
  functional can be uniquely identified with the trace of a function $\TTheta\in\cH(\div\Div\!,T)$
  with $\Grad\grad \div\Div\TTheta + \TTheta = 0$.
  Let $\alpha,\beta\in P^0(\cE_T)$ be the functions associated
  to~\eqref{eq:dofDDlocal_a} and \eqref{eq:dofDDlocal_b},
  that is, $\alpha|_E:=\alpha_E |E|^{-1}$ and $\beta|_E:=\beta_E|E|^{-1}$ ($E\in\cE_T$). 
  For $z\in H^2(T)$ we set
  \begin{align} \label{disc_duality}
    \ell(z) := \dual{\alpha}{z}_{\partial T} - \dual{\beta}{\nn\cdot\grad z}_{\partial T}
    - \sum_{e\in\cN_T} \gamma_e z(e)
  \end{align}
  with $\gamma_e$ ($e\in\cN_T$) as in \eqref{eq:dofDDlocal_c}.
  Note that $\ell(z) = 0$ if $z\in H_0^2(T)$ and that $\ell(\cdot)$ is indeed a bounded functional.
  By selecting appropriate test functions $z\in H^2(T)$ we obtain $\ell(\cdot)=\traceDD{T}(\TTheta)$ 
  Furthermore, we prove that
  \begin{align*}
    \ell(z) = 0 \quad\text{for all } z\in H^2(T) \Leftrightarrow (\alpha,\beta,\gamma) = 0.
  \end{align*}
  The direction ``$\Leftarrow$'' is trivial and ``$\Rightarrow$'' follows by selecting appropriate test functions.
  Note that $\ell(\cdot)=0$ implies $\traceDD{T}(\TTheta) = 0$ with $\Grad\grad\div\Div\TTheta + \TTheta=0$.
  Since $\norm{\TTheta}{\div\Div\!,T} = \norm{\traceDD{T}(\TTheta)}{-3/2,-1/2,\partial T}$ it follows $\TTheta=0$.

  Finally, to see the other direction, let $\TTheta\in U_{\trddiv T}$ be given.
  Note that $\TTheta$ only depends on its trace values and by the localization of traces from
  \S\ref{sec_jump2_dd} we conclude that $\dim(U_{\trddiv T}) = 9$
  which is the number of degrees of freedom~\eqref{eq:dofDDlocal} and, thus, finishes the proof.
\end{proof}

The corresponding global (volume) space is defined by
\begin{align*}
  U_{\trddiv\cT} := \{\TTheta\in \cH(\div\Div\!,\Omega);\; \TTheta|_{T}\in U_{\trddiv{T}}\ \forall T\in\cT\}
\end{align*}
with associated degrees of freedom
\begin{subequations}\label{eq:dofDDglobal}
\begin{align}
  \dual{\nn\cdot\Div\TTheta + \partial_\bt (\bt\cdot\TTheta\nn)}{1}_E
  &\quad (E\in\cE), \label{eq:dofDDglobal:a} \\
  \dual{\nn\cdot\TTheta\nn}{1}_E
  &\quad (E\in\cE), \label{eq:dofDDglobal:b}\\
  \jjump{\TTheta}_{\partial T}(e)
  &\quad (e\in \cN_T,\  T\in\cT),\label{eq:dofDDglobal:c} \\
  \text{subject to }
  \sum_{T\in \omega(e)} \jjump{\TTheta}_{\partial T}(e) = 0
  &\quad\forall e\in\cN_0.
  \label{eq:dofDDglobal:d}
\end{align}
\end{subequations}
Analogously to \eqref{eq:dofDDlocal},~\eqref{disc_duality}, these variables define a functional
acting on $z\in H^2(\cT)$.

\begin{lemma}\label{lem:dofDDglobal}
  The degrees of freedom~\eqref{eq:dofDDglobal} uniquely define an element in $U_{\trddiv\cT}$.
\end{lemma}
\begin{proof}
  Note that by Lemma~\ref{lem:dofDDlocal},~\eqref{eq:dofDDglobal:a}--\eqref{eq:dofDDglobal:c}
  defines a unique function $\TTheta\in H(\div\Div\!,\cT)$.
  Proposition~\ref{prop_cor_dd_jump} and~\eqref{eq:dofDDglobal:d} conclude the proof.
\end{proof}
Summing up, $U_{\trddiv\cT}$ has $\#\cE+\#\cE+3\#\cT-\#\cN_0$ degrees of freedom.
In the implementation we take care of the constraints~\eqref{eq:dofDDglobal:d}
by using Lagrange multipliers. Now, for the approximation of
$\tQ\in\bH^{-3/2,-1/2}(\cS)$, we use the discrete space
\begin{align*}
  \widehat Q_{\cS} := \traceDD{} (U_{\trddiv\cT}).
\end{align*}
By Proposition~\ref{prop_dd_trace} there is an isomorphism between the volume space $U_{\trddiv\cT}$ and its
trace $\widehat Q_{\cS}$ (note the PDE-constraint in \eqref{eq:defUddivT}).
Therefore, the trace space $\widehat Q_{\cS}$ has the same degrees of freedom \eqref{eq:dofDDglobal}.

\begin{lemma}\label{lem:approxTraceDD}
  Let $u\in H^4(\Omega)$  and set
  $\tQ:=\traceDD{}(\cC\Grad\grad u)$. Then, 
  \begin{align*}
    \min_{\tQ_h\in \widehat Q_{\cS}} \norm{\tQ-\tQ_h}{-3/2,-1/2,\cS} \leq C h \norm{u}{H^4(\Omega)}
  \end{align*}
  where the generic constant $C>0$ only depends on the shape-regularity constant
  $C_\mathrm{shape}$ of $\cT$, and $\cC$.
\end{lemma}
\begin{proof}
  Set $\TTheta := \cC\Grad\grad u$.
  We start with defining an element $\tQ_h \in \widehat Q_\cS$.
  Let $T\in\cT$ be given and let $\Pi^p:\,L_2(\partial T)\to P^p(\cE_T)$ denote the $L_2$-projection.
  Below, this projection operator will also be used component-wise for vector functions.
  For $E\in\cE_T$ we set (cf.~\eqref{traces_tn})
  $\phi_E := \traceDD{T,E,\bt}(\TTheta)
           = \bigl(\nn\cdot\Div\TTheta|_T + \partial_\bt(\bt\cdot\TTheta|_T\nn)\bigr)|_E$,
  $\psi_E:= \traceDD{T,E,\nn}(\TTheta)
           =\bigl(\nn\cdot\TTheta|_T\nn\bigr)|_E$
  and define $\phi,\psi\in L_2(\partial T)$ by $\phi|_E:=\phi_E$ and $\psi|_E:=\psi_E$ for $E\in\cE_T$.

  By the regularity assumption we even have $\phi_E\in H^1(T)|_E$.
  Thus, there exists (a more regular) antiderivative $g_E$, that is,
  $\partial_\bt g_E = \phi_E$, and it satisfies
  \begin{align*}
    \dual{\phi_E}{z}_E = -\dual{g_E}{\partial_\bt z}_E + g_E(e_+)z(e_+) - g_E(e_-)z(e_-),
  \end{align*}
  where $E$ is the edge with vertices $e_\pm$.
  Define $g\in L_2(\partial T)$ by $g|_E = g_E$ with jumps
  $\jjump{g}_{\partial T}(e) := g|_{E_2}(e)-g|_{E_1}(e)$ for $e\in\cN_T$.
  Here, $E_1,E_2\in \cE_T$ are the unique edges with $\overline E_1\cap\overline E_2 = \{e\}$
  and the sign is chosen to
  be consistent with the definition of $\jjump{\TTheta}_{\partial T}(e)$, cf.~\eqref{jjump_node}.
  We set $\gamma_e := \jjump{\TTheta}_{\partial T}(e) - \jjump{g-\Pi^1g}_{\partial T}(e)$ for $e\in\cN_T$.
  Prescribing the values of the degrees of freedom~\eqref{eq:dofDDlocal} as
  $\partial_{\bt,\cE_T} (\Pi^1 g) \in P^0(\cE_T)$, $\Pi^0\psi\in P^0(\cE_T)$, 
  and $(\gamma_e)_{e\in\cN_T}\in\R^3$, this defines a unique element of $U_{\trddiv{T}}$.
  Doing this for all elements $T\in\cT$ we obtain a unique element of $U_{\trddiv{\cT}}$ since
  \begin{align*}
    \sum_{T\in\omega(e)} \jjump{\TTheta}_{\partial T}(e) - \jjump{g-\Pi^1g}_{\partial T}(e) = 0
    \quad\forall e\in\cN_0.
  \end{align*}
  This also defines an element in $\widehat Q_\cS$ which we denote by $\tQ_h$.

  To analyze the convergence order it suffices to do so for one element $T\in\cT$.
  Let $\widetilde\TTheta\in \cH(\div\Div\!,T)$ be the unique element with
  $\traceDD{T}(\widetilde\TTheta) = (\tQ-\tQ_h)|_{\partial T}$ and
  $\Grad\grad\div\Div\widetilde\TTheta + \widetilde\TTheta = 0$ on $T$. 
  The proof of Lemma~\ref{la_tr_dd_norms} shows that
  \begin{align*}
    \norm{(\tQ-\tQ_h)|_{\partial T}}{-3/2,-1/2,\partial T}^2
    = \norm{\widetilde\TTheta}{T}^2+\norm{\div\Div\widetilde\TTheta}T^2
    = \dual{\traceDD{T}(\widetilde\TTheta)}z_{\partial T}
  \end{align*}
  where $z := -\div\Div\widetilde \TTheta \in H^2(T)$ and $\norm{z}{2,T}^2 =
  \norm{\widetilde\TTheta}{T}^2+\norm{\div\Div\widetilde\TTheta}T^2$.
  Note that (cf.~\eqref{pf_cor_prop} and Remark~\ref{rem:traceOp2D})
  \begin{align} \label{pf_la_approxTraceDD}
    \dual{\traceDD{T}(\widetilde\TTheta)}{z}_{\partial T}
    = \dual{\partial_{\bt,\cE_T} (1-\Pi^1)g}z_{\partial T}
    - \sum_{e\in\cN_T} \jjump{g-\Pi^1g}_{\partial T}(e)z(e)
    - \dual{(1-\Pi^0)\psi}{\nn\cdot\grad z}_{\partial T}
  \end{align}
  by the definition of $\gamma_e$. The last term in \eqref{pf_la_approxTraceDD} is estimated by
  \begin{align*}
    |\dual{(1-\Pi^0)\psi}{\nn\cdot\grad z}_{\partial T}| 
    &= |\dual{(1-\Pi^0)\psi}{(1-\Pi^0)\nn\cdot\grad z}_{\partial T}|
    \le \norm{(1-\Pi^0)\psi}{\partial T}\norm{(1-\Pi^0)\grad z}{\partial T} \\
    &\lesssim h \norm{u}{H^3(T)} \norm{\Grad\grad z}{T}
    \leq h \norm{u}{H^3(T)} \norm{\widetilde\TTheta}{\div\Div\!,T}.
  \end{align*}
  Here we have used the trace inequality $\norm{(1-\Pi^0)v|_E}{E} \lesssim h^{1/2}\norm{\grad v}{T}$
  for $E\in\cE_T$ and $v\in H^1(T)$. The involved constants only depend on the shape-regularity of $\cT$.
  To bound the two remaining terms in \eqref{pf_la_approxTraceDD} we integrate by parts,
  use properties of the $L_2$ projection $\Pi^1$ and the trace inequality
  $\norm{(1-\Pi^1)\grad z}{\partial T}\lesssim h^{1/2}\norm{\Grad\grad z}{T}$. This yields
  \begin{align*}
  \lefteqn{
    \dual{\partial_{\bt,\cE_T} (1-\Pi^1)g}z_{\partial T}
    - \sum_{e\in\cN_T} \jjump{g-\Pi^1g}_{\partial T}(e)z(e)
    =
    -\dual{(1-\Pi^1)g}{\partial_\bt z}_{\partial T}
  }\\
    &=
    -\dual{(1-\Pi^1)g}{(1-\Pi^1)\partial_\bt z}_{\partial T}
    \le \norm{(1-\Pi^1)g}{\partial T}\norm{(1-\Pi^1)\grad z}{\partial T}
    \lesssim
    h^{3/2}\norm{\partial_{\bt,\cE_T} g}{\partial T} \norm{\Grad\grad z}{T}.
  \end{align*}
  With the trace inequality $\norm{v|_E}{E} \lesssim h^{-1/2}\norm{v}{H^1(T)}$ for $E\in\cE_T$ and
  $v\in H^1(T)$ we have that
  \[
     h\norm{\partial_{\bt,\cE_T} g}{\partial T}^2
     =
     h\norm{\phi}{\partial T}^2
     =
     h \sum_{E\in\cE_T} \|\traceDD{T,E,\bt}(\TTheta)\|_E^2
     \lesssim
     \norm{u}{H^4(T)}.
  \]
  Therefore, eventually we obtain the desired bound for the remaining terms in \eqref{pf_la_approxTraceDD},
  \begin{align*}
    \Bigl\lvert\dual{\partial_{\bt,\cE_T} (1-\Pi^1)g}z_{\partial T}
               - \sum_{e\in\cN_T} \jjump{g-\Pi^1g}_{\partial T}(e)z(e)
    \Bigr\rvert 
    \lesssim
    h \norm{u}{H^4(T)}\norm{z}{2,T}.
  \end{align*}
  Altogether we have thus shown that
  \begin{align*}
    \norm{(\tQ-\tQ_h)|_{\partial T}}{-3/2,-1/2,\partial T}^2 = \norm{\widetilde\TTheta}{\div\Div\!,T}^2
    \lesssim h \norm{u}{H^4(T)} \norm{\widetilde\TTheta}{\div\Div\!,T},
  \end{align*}
  that is,
  $\norm{(\tQ-\tQ_h)|_{\partial T}}{-3/2,-1/2,\partial T}^2 \lesssim h^2 \norm{u}{H^4(T)}^2$.
  Summation over all $T\in\cT$ finishes the proof.
\end{proof}

Our final discrete subspace of $\UU$ for the DPG approximation is
\begin{align*}
  \UU_h := P^0(\cT) \times (\PP^0(\cT)\cap \LL_2^s(\Omega)) \times
  \widehat U_\cS \times \widehat Q_\cS.
\end{align*}
By standard results on the approximation properties of $P^0(T)$ in $L_2(T)$, Remark~\ref{rem:approxTraceGG} 
and Lemma~\ref{lem:approxTraceDD} we get the following result:
\begin{theorem}\label{thm:approxU}
  Let $u\in H^4(\Omega)$ and set
  $\uu := (u,\cC\Grad\grad u, \traceGG{}(u), \traceDD{}(\cC\Grad\grad u))\in \UU$.
  Then, it holds
  \begin{align*}
    \min_{\bw_h\in\UU_{h}}\norm{\uu-\bw_h}\UU \leq C h \norm{u}{H^4(\Omega)}.
  \end{align*}
  The constant $C>0$ depends on the shape-regularity of $\cT$ and $\cC$,
  but is otherwise independent of $\cT$.\qed
\end{theorem}
\begin{remark}
  Let us note that the regularity assumption $u\in H^4(\Omega)$ in Theorem~\ref{thm:approxU} may be reduced
  to $u\in H^3(\Omega)$ subject to $\div\Div\cC\Grad\grad u\in L_2(\Omega)$
  with a refined analysis of Lemma~\ref{lem:approxTraceDD}.
  Such a reduction in the regularity assumption was observed in the recent work~\cite{Fuehrer_18_SDM} for 
  ultra-weak formulations of second order elliptic problems.
\end{remark}

Since the optimal test functions cannot be computed exactly, we approximate them in the enlarged space
\begin{align*}
  \VV_h = P^3(\cT) \times (\PP^2(\cT)\cap \LL_2^s(\Omega)) \subset \VV.
\end{align*}
That is, we replace $\ttt:\;\UU\to\VV$ by $\ttt_h:\;\UU\to\VV_h$, which is defined by
\begin{align*}
  \ip{\ttt_h\uu}{\vv}_\VV = b(\uu,\vv) \quad\forall \vv\in\VV_h.
\end{align*}
Particularly, the space of approximated discrete optimal test functions is given by
$\ttt_h(\UU_h)\subseteq \VV_h$.

%====================================
% EXAMPLES
%====================================
\subsection{Examples} \label{sec_num_ex}
In the following two examples, refinements are obtained by
using the newest vertex bisection (NVB). It maintains shape-regularity of the triangulation, i.e.,
\begin{align*}
  \sup_{T\in\cT} \frac{\diam(T)^2}{|T|} \leq C \sup_{T\in\cT_0} \frac{\diam(T)^2}{|T|}
\end{align*}
where $C>0$ is independent of $\cT$, and $\cT$ is an arbitrary refinement of the initial mesh $\cT_0$.
Uniform refinement means that each triangle is divided into four son triangles with the same area, i.e., 
it corresponds to two bisections of the father element.
In the second example we use a simple adaptive loop of the form
\begin{align*}
  \boxed{\texttt{SOLVE}} \quad\longrightarrow\quad
  \boxed{\texttt{ESTIMATE}} \quad\longrightarrow\quad
  \boxed{\texttt{MARK}} \quad\longrightarrow\quad
  \boxed{\texttt{REFINE}} \quad.
\end{align*}
The estimation step is done with the error estimator that is automatically provided by the DPG method, 
$\eta := \norm{B(\uu-\uu_h)}{\VV_h'}$. 
We refer to~\cite{CarstensenDG_14_PEC} for an abstract analysis of the DPG error estimator.
Let us note that $\eta$ can be written as the sum of local contributions
\begin{align*}
  \eta^2 = \sum_{T\in\cT} \eta(T)^2.
\end{align*}
The marking step is done using the bulk criterion ($\theta \in (0,1)$)
\begin{align*}
  \theta \eta^2 \leq \sum_{T\in\mathcal{M}} \eta(T)^2
\end{align*}
where $\mathcal{M}\subseteq\cT$ is the set of marked elements. It is the set of (up to a constant) minimal
cardinality that satisfies the above relation.
In \S\ref{sec:Zshape} we use the parameter $\theta = \tfrac12$.

\subsubsection{Square domain}\label{sec:quad}
Let $\Omega = (0,1)^2$. We use the constant load $f=1$, the identity $\cC=I$, and the boundary conditions
\begin{align*}
  u|_{\partial \Omega} = 0, \quad \nn\cdot\MM\nn|_{\partial \Omega} = 0.
\end{align*}
It is known that the exact solution can be expressed by the double Fourier series
\begin{align*}
  u(x,y) = \frac{16}{\pi^2} 
  \sum_{n=0}^\infty \sum_{m=0}^\infty \frac{\sin((2n+1)\pi x)\sin((2m+1)\pi y)}{(2n+1)(2m+1)((2n+1)^2+(2m+1)^2)^2}.
\end{align*}
\begin{figure}[htb]
  \begin{center}
    \includegraphics{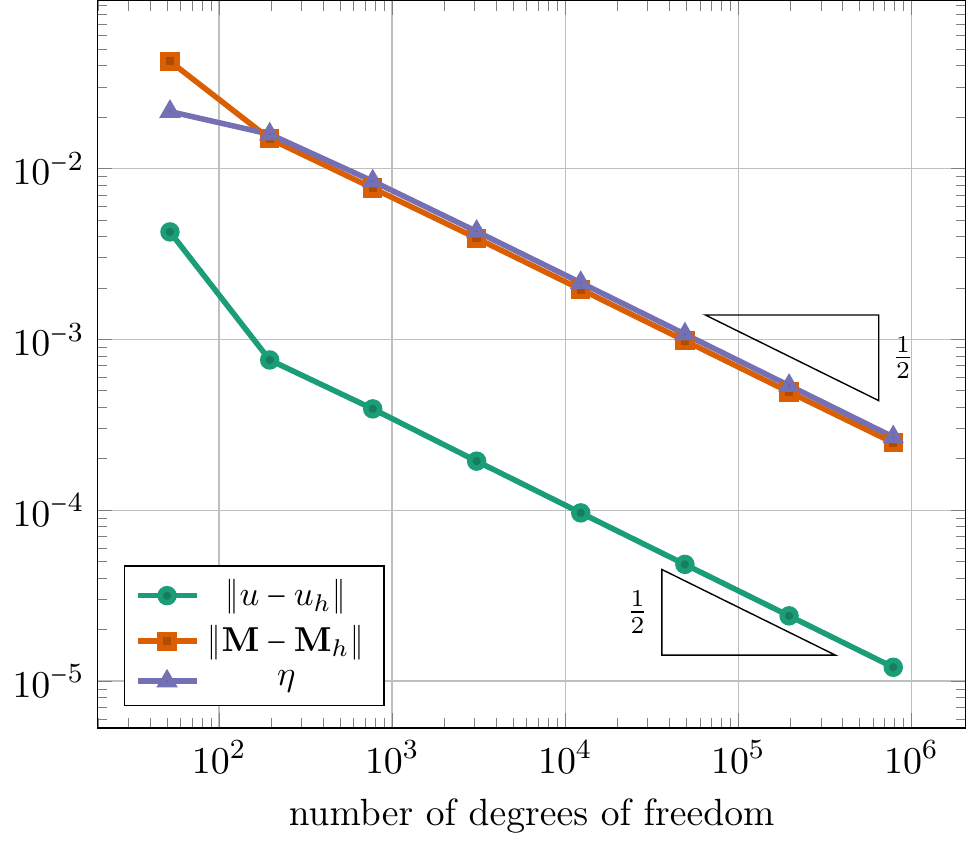}
  \end{center}
  \caption{$L_2$ error for the field variables $u$, $\MM$, and the DPG error estimator
           $\eta$ with respect to the degrees of freedom (\S\ref{sec:quad}).}
  \label{fig:quadRates}
\end{figure}
In particular, the solution is smooth and we therefore expect a convergence of order $\OO(h)$.
In order to compute the $L_2(\Omega)$ errors $\norm{u-u_h}{}$ and $\norm{\MM-\MM_h}{}$ we replace the Fourier series by
finite sums, 
\begin{align*}
  u(x,y) \approx \frac{16}{\pi^2} 
  \sum_{n=0}^{15} \sum_{m=0}^{15} \frac{\sin((2n+1)\pi x)\sin((2m+1)\pi y)}{(2n+1)(2m+1)((2n+1)^2+(2m+1)^2)^2}.
\end{align*}
Figure~\ref{fig:quadRates} shows the convergence behavior of the $L_2$ errors and the DPG error estimator $\eta$ with
respect to the number of degrees of freedom ($=\dim(\UU_\cT)$) for a sequence of uniformly refined meshes.
The number $\alpha>0$ besides the triangle in the plots indicates its negative slope, i.e.,
the hypotenuse is parallel to $\dim(\UU_\cT)^{-\alpha}$.
We observe that all the plotted quantities have the same order of convergence $\alpha = 1/2$. 
Note that by \S\ref{sec:discretespaces} we have $\dim(\UU_\cT) \simeq \#\cT \simeq h^{-2}$.
Hence, we see the optimal convergence behavior $\OO(h)$ as stated in Theorem~\ref{thm:approxU}.

\subsubsection{Domain with reentrant corner}\label{sec:Zshape}
We consider the non-convex domain with reentrant corner at $(x,y)=(0,0)$ visualized in Figure~\ref{fig:ZshapeMeshes}
with angle $\tfrac3{4}\pi$ between the two edges that meet at $(x,y)=(0,0)$.
We use the singularity function
\begin{align*}
  u(r,\varphi) = r^{1+\alpha}(\cos( (\alpha+1)\varphi)+C \cos( (\alpha-1)\varphi))
\end{align*}
with polar coordinates $(r,\varphi)$ centered at the origin. A straightforward calculation yields
\begin{align*}
  \div\Div\Grad\grad u = 0 =: f.
\end{align*}
For the boundary conditions we prescribe the values of $u|_\Gamma$ and $\nabla u|_\Gamma$.
The parameters $\alpha$ and $C$ are chosen such that $u$ and its normal derivative vanish on the boundary edges
that meet at the origin.
Here, we have $\alpha\approx 0.673583432147380$ and $C\approx 1.234587795273723$.
Note that $u\in H^{2+\alpha-\varepsilon}(\Omega)$ and, selecting $\cC=I$, 
$\MM = \Grad\grad u \in (H^{\alpha-\varepsilon}(\Omega))^{2\times 2}$ for $\varepsilon>0$.
Furthermore, one verifies that $|\Div\MM(r,\varphi)|\simeq r^{\alpha-2}\not\in L_2(\Omega)$.
Therefore, $\MM\in H(\div\Div\!,\Omega)$ and $\MM\notin \bH(\Div\!,\Omega)$
(recall our discussion in Remark~\ref{rem_traceDD}).
However, $\MM\in\cH(\div\Div\!,\Omega)$ as can be seen as follows.
Let $E$ denote one of the boundary edges with endpoint $(0,0)$. Then,
$\nn\cdot\MM\nn|_E \simeq r^{\alpha-1} \in L_2(E)$ and $\bt\cdot\MM\nn|_E = 0$.
Moreover, $\nn\cdot\Div\MM|_E \simeq r^{\alpha-2} \in (H^1(E))'$. To see the last claim we note
that $r^{\alpha-2}\simeq (r^{\alpha-1})'$, 
where $(\cdot)'$ denotes the generalized derivative operator.
In particular, the latter is bounded as a mapping from $L_2(E)\to (H^1(E))'$, 
see, e.g.,~\cite[Proof of Lemma~3.5]{StephanS_91_hpB} and therefore,
$\norm{\nn\cdot\Div\MM}{(H^{1}(E))'}\lesssim \norm{r^{\alpha-1}}{L_2(E)}<\infty$.

\begin{figure}[htb]
  \begin{center}
    \includegraphics{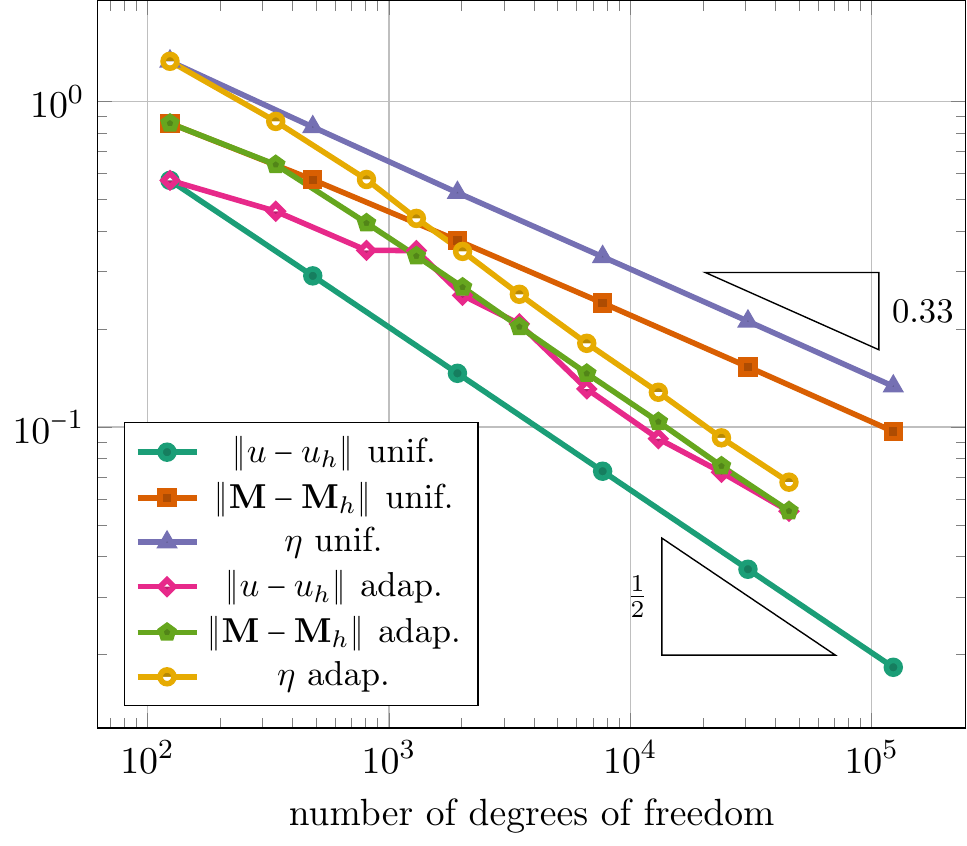}
  \end{center}
  \caption{Convergence rates of the DPG error estimator for uniformly and adaptively refined meshes
    (\S\ref{sec:Zshape}).}
  \label{fig:ZshapeRates}
\end{figure}

Due to the reduced regularity of $\MM$, uniform mesh refinements will lead to a suboptimal convergence order
$\OO(h^\alpha) = \OO(\dim(\UU_h)^{-\alpha/2})$.
In Figure~\ref{fig:ZshapeRates} we plot the DPG error estimator $\eta$ and the $L_2$ errors of the field variables 
in the case of uniform and adaptive mesh refinements.
We observe that uniform refinements lead indeed to a suboptimal convergence rate whereas with our adaptive algorithm the
optimal rates $\OO(\dim(\UU_h)^{-1/2})$ are restored for the error estimator and $\norm{\MM-\MM_h}{}$.

Figure~\ref{fig:ZshapeMeshes} shows meshes obtained from the adaptive algorithm in the iterations $j=1,2,3,4$.
We observe a strong refinement towards the reentrant corner where the (higher order) derivatives of $u$ are singular.

\begin{figure}[htb]
  \begin{center}
    \includegraphics{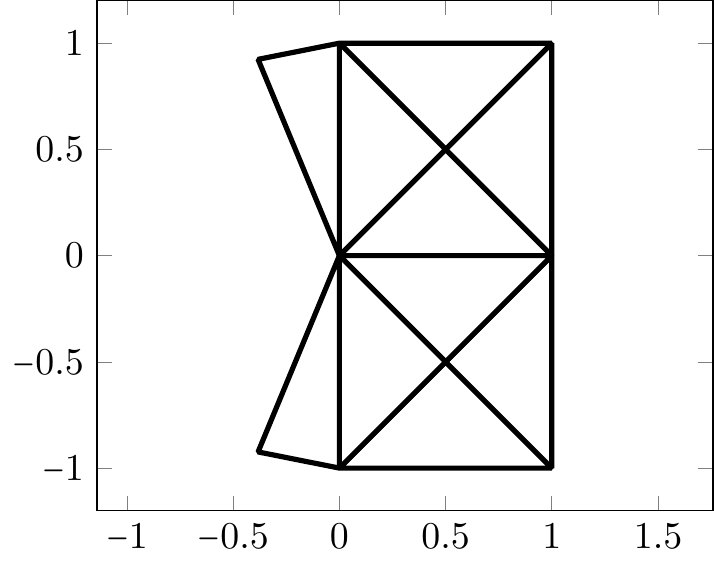}
    \includegraphics{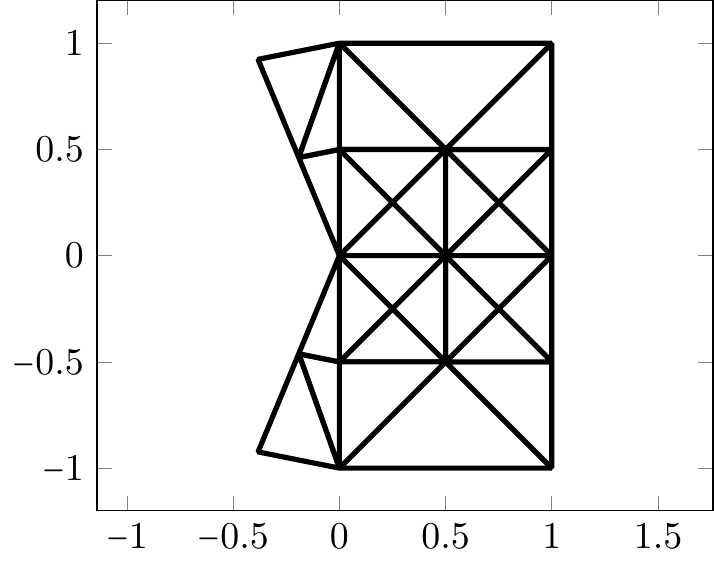}
    \includegraphics{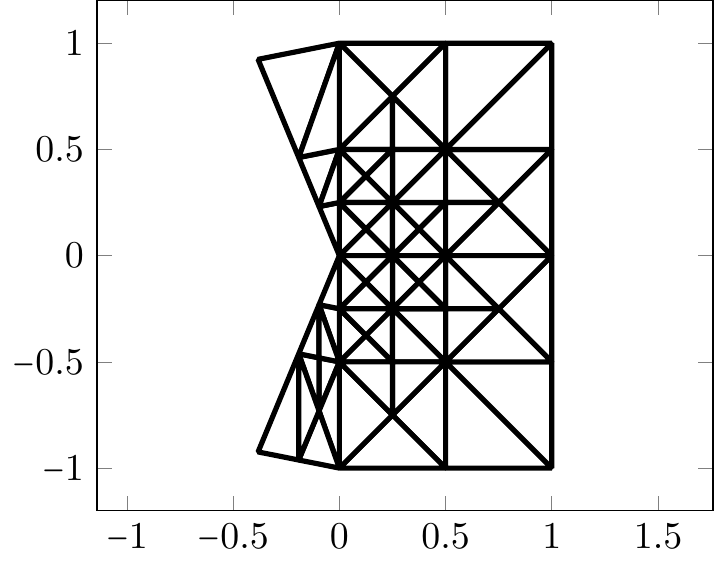}
    \includegraphics{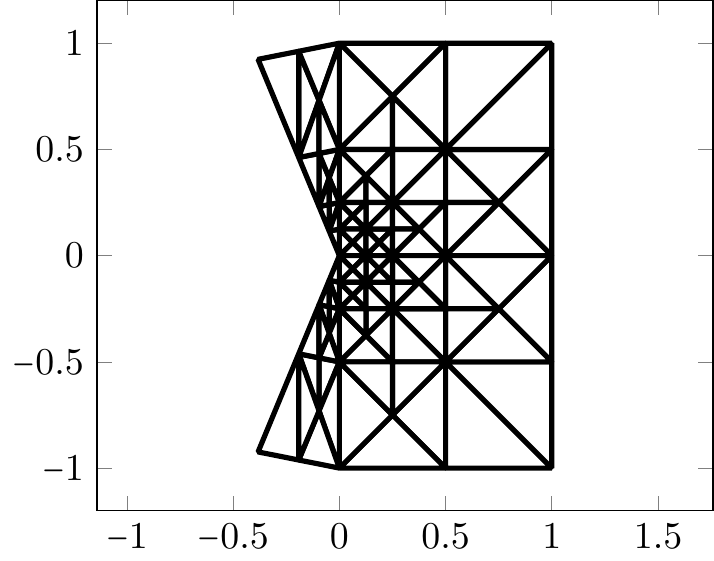}
  \end{center}
  \caption{Meshes obtained from the adaptive algorithm (iterations $j=1,2,3,4$) 
  with $\#\cT = 10, 28, 67, 108$ elements.}
  \label{fig:ZshapeMeshes}
\end{figure}

%===================================================================================================
\bibliographystyle{siam}
%\bibliography{paper}
\bibliography{/home/norbert/tex/bib/heuer,/home/norbert/tex/bib/bib}

\begin{thebibliography}{10}

\bibitem{AmaraCPC_02_BMM}
{\sc M.~Amara, D.~Capatina-Papaghiuc, and A.~Chatti}, {\em Bending moment mixed
  method for the {Kirchhoff--Love} plate model}, SIAM J. Numer. Anal., 40
  (2002), pp.~1632--1649.

\bibitem{BeiraodaVeigaNS_10_PEE}
{\sc L.~Beir\~ao~da Veiga, J.~Niiranen, and R.~Stenberg}, {\em A posteriori
  error analysis for the {M}orley plate element with general boundary
  conditions}, Internat. J. Numer. Methods Engrg., 83 (2010), pp.~1--26.

\bibitem{BramwellDGQ_12_Lhp}
{\sc J.~Bramwell, L.~Demkowicz, J.~Gopalakrishnan, and W.~Qiu}, {\em A
  locking-free {$hp$} {DPG} method for linear elasticity with symmetric
  stresses}, Numer. Math., 122 (2012), pp.~671--707.

\bibitem{BrezziM_13_VEM}
{\sc F.~Brezzi and L.~D. Marini}, {\em Virtual element methods for plate
  bending problems}, Comput. Methods Appl. Mech. Engrg., 253 (2013),
  pp.~455--462.

\bibitem{BroersenS_14_RPG}
{\sc D.~Broersen and R.~Stevenson}, {\em A robust {Petrov}-{Galerkin}
  discretisation of convection-diffusion equations}, Comput. Math. Appl., 68
  (2014), pp.~1605--1618.

\bibitem{BroersenS_15_PGD}
\leavevmode\vrule height 2pt depth -1.6pt width 23pt, {\em A
  {Petrov}-{Galerkin} discretization with optimal test space of a mild-weak
  formulation of convection-diffusion equations in mixed form}, IMA J. Numer.
  Anal., 35 (2015), pp.~39--73.

\bibitem{BuffaCS_02_THL}
{\sc A.~Buffa, M.~Costabel, and D.~Sheen}, {\em On traces for {H}(curl,
  {$\Omega$}) in {Lipschitz} domains}, J. Math. Anal. Appl., 276 (2002),
  pp.~845--867.

\bibitem{CaloCN_14_ADP}
{\sc V.~M. Calo, N.~O. Collier, and A.~H. Niemi}, {\em Analysis of the
  discontinuous {P}etrov-{G}alerkin method with optimal test functions for the
  {R}eissner-{M}indlin plate bending model.}, Comput. Math. Appl., 66 (2014),
  pp.~2570--2586.

\bibitem{Carstensen_02_RBP}
{\sc C.~Carstensen}, {\em Residual-based a posteriori error estimate for a
  nonconforming {R}eissner-{M}indlin plate finite element}, SIAM J. Numer.
  Anal., 39 (2002), pp.~2034--2044.

\bibitem{CarstensenDG_14_PEC}
{\sc C.~Carstensen, L.~Demkowicz, and J.~Gopalakrishnan}, {\em A posteriori
  error control for {DPG} methods}, SIAM J. Numer. Anal., 52 (2014),
  pp.~1335--1353.

\bibitem{CarstensenDG_16_BSF}
{\sc C.~Carstensen, L.~Demkowicz, and J.~Gopalakrishnan}, {\em Breaking spaces
  and forms for the {DPG} method and applications including {M}axwell
  equations}, Comput. Math. Appl., 72 (2016), pp.~494--522.

\bibitem{ChanHBTD_14_RDM}
{\sc J.~Chan, N.~Heuer, T.~Bui-Thanh, and L.~Demkowicz}, {\em Robust {DPG}
  method for convection-dominated diffusion problems {II}: {Adjoint} boundary
  conditions and mesh-dependent test norms}, Comput. Math. Appl., 67 (2014),
  pp.~771--795.

\bibitem{Ciarlet_78_IEE}
{\sc P.~G. Ciarlet}, {\em Interpolation error estimates for the reduced
  {H}sieh-{C}lough-{T}ocher triangle}, Math. Comp., 32 (1978), pp.~335--344.

\bibitem{DemkowiczG_11_ADM}
{\sc L.~Demkowicz and J.~Gopalakrishnan}, {\em Analysis of the {DPG} method for
  the {Poisson} problem}, SIAM J. Numer. Anal., 49 (2011), pp.~1788--1809.

\bibitem{DemkowiczG_11_CDP}
\leavevmode\vrule height 2pt depth -1.6pt width 23pt, {\em A class of
  discontinuous {Petrov-Galerkin} methods. {Part II}: {O}ptimal test
  functions}, Numer. Methods Partial Differential Eq., 27 (2011), pp.~70--105.

\bibitem{DemkowiczG_13_PDM}
\leavevmode\vrule height 2pt depth -1.6pt width 23pt, {\em A primal {DPG}
  method without a first-order reformulation}, Comput. Math. Appl., 66 (2013),
  pp.~1058--1064.

\bibitem{DemkowiczGNS_17_SDM}
{\sc L.~Demkowicz, J.~Gopalakrishnan, S.~Nagaraj, and P.~Sep{\'u}lveda}, {\em A
  spacetime {DPG} method for the {Schr\"odinger} equation}, SIAM J. Numer.
  Anal., 55 (2017), pp.~1740--1759.

\bibitem{DemkowiczGN_12_CDP}
{\sc L.~Demkowicz, J.~Gopalakrishnan, and A.~H. Niemi}, {\em A class of
  discontinuous {P}etrov-{G}alerkin methods. {P}art {III}: {A}daptivity}, Appl.
  Numer. Math., 62 (2012), pp.~396--427.

\bibitem{DemkowiczH_13_RDM}
{\sc L.~Demkowicz and N.~Heuer}, {\em Robust {DPG} method for
  convection-dominated diffusion problems}, SIAM J. Numer. Anal., 51 (2013),
  pp.~2514--2537.

\bibitem{Ernesti_thesis}
{\sc J.~Ernesti}, {\em {Space-Time Methods for Acoustic Waves with Application
  to Full Waveform Inversion}}, PhD thesis, Karlsruhe Institute of Technology,
  Karlsruhe, Germany, 2017.

\bibitem{ErnestiW_STD}
{\sc J.~Ernesti and C.~Wieners}, {\em A space-time {DPG} method for acoustic
  waves}, manuscript.
\newblock To appear in \emph{Radon Series on Computational and Applied
  Mathematics}.

\bibitem{Fuehrer_18_SDM}
{\sc T.~F{\"u}hrer}, {\em Superconvergence in a {DPG} method for an ultra-weak
  formulation}, Comput. Math. Appl., 75 (2018), pp.~1705--1718.

\bibitem{FuehrerH_17_RCD}
{\sc T.~F{\"u}hrer and N.~Heuer}, {\em Robust coupling of {DPG} and {BEM} for a
  singularly perturbed transmission problem}, Comput. Math. Appl., 74 (2017),
  pp.~1940--1954.

\bibitem{FuehrerHS_DMS}
{\sc T.~F{\"u}hrer, N.~Heuer, and E.~P. Stephan}, {\em On the {DPG} method for
  {Signorini} problems}, IMA J. Numer. Anal.
\newblock (in press, DOI:10.1093/imanum/drx048).

\bibitem{Gallistl_17_SSP}
{\sc D.~Gallistl}, {\em Stable splitting of polyharmonic operators by
  generalized {S}tokes systems}, Math. Comp., 86 (2017), pp.~2555--2577.

\bibitem{GopalakrishnanS_SDM}
{\sc J.~Gopalakrishnan and P.~Sep{\'u}lveda}, {\em A spacetime {DPG} method for
  acoustic waves}, {arXiv}: 1709.08268, 2017.

\bibitem{HansboL_11_PEE}
{\sc P.~Hansbo and M.~G. Larson}, {\em A posteriori error estimates for
  continuous/discontinuous {G}alerkin approximations of the {K}irchhoff-{L}ove
  plate}, Comput. Methods Appl. Mech. Engrg., 200 (2011), pp.~3289--3295.

\bibitem{HausslerCombe_15_CMR}
{\sc U.~H\"{a}ussler-Combe}, {\em Computational Methods for Reinforced Concrete
  Structures}, Ernst \& Sohn, Berlin, Germany, 2015.

\bibitem{HeuerK_17_DPG}
{\sc N.~Heuer and M.~Karkulik}, {\em Discontinuous {Petrov}-{Galerkin} boundary
  elements}, Numer. Math., 135 (2017), pp.~1011--1043.

\bibitem{HeuerK_17_RDM}
\leavevmode\vrule height 2pt depth -1.6pt width 23pt, {\em A robust {DPG}
  method for singularly perturbed reaction-diffusion problems}, SIAM J. Numer.
  Anal., 55 (2017), pp.~1218--1242.

\bibitem{KeithFD_17_DMA}
{\sc B.~Keith, F.~Fuentes, and L.~Demkowicz}, {\em The {DPG} methodology
  applied to different variational formulations of linear elasticity}, Comput.
  Methods Appl. Mech. Eng., 309 (2017), pp.~579--609.

\bibitem{Kirchhoff_50_UGB}
{\sc G.~R. Kirchhoff}, {\em \"{U}ber das {G}leichgewicht und die {B}ewegung
  einer elastischen {S}cheibe}, J. Math (Crelle), 50 (1850), pp.~51--88.

\bibitem{Love_88_SFV}
{\sc A.~E.~H. Love}, {\em The small free vibrations and deformation of a thin
  elastic shell}, Phil. Trans. R. Soc. Lond. A, 179 (1888), pp.~491--546.

\bibitem{NiemiBD_11_DPG}
{\sc A.~H. Niemi, J.~Bramwell, and L.~Demkowicz}, {\em Discontinuous
  {P}etrov-{G}alerkin method with optimal test functions for thin-body problems
  in solid mechanics}, Comput. Methods Appl. Mech. Eng., 200 (2011),
  pp.~1291--1300.

\bibitem{RafetsederZ_DRK}
{\sc K.~Rafetseder and W.~Zulehner}, {\em A decomposition result for
  {Kirchhoff} plate bending problems and a new discretization approach},
  {arXiv}: 1703.07962, 2017.

\bibitem{StephanS_91_hpB}
{\sc E.~P. Stephan and M.~Suri}, {\em The $h$-$p$ version of the boundary
  element method on polygonal domains with quasiuniform meshes}, RAIRO Mod\'el.
  Math. Anal. Num\'er., 25 (1991), pp.~783--807.

\bibitem{SzaboA_12_SGT}
{\sc B.~Szab\'o and R.~Actis}, {\em Simulation governance: {T}echnical
  requirements for mechanical design}, Comput. Methods Appl. Mech. Eng.,
  249-252 (2012), pp.~158--168.

\bibitem{VentselK_01_TPS}
{\sc E.~Ventsel and T.~Krauthammer}, {\em Thin Plates and Shells}, CRC Press,
  New York, 2001.

\end{thebibliography}
\end{document}